\documentclass[a4paper,10pt]{article}

 \usepackage[latin1]{inputenc}
\usepackage{psfrag}

\usepackage{setspace}
\usepackage[totalwidth=14cm,totalheight=24cm]{geometry}
\usepackage{epsfig}
\usepackage{amsmath,amssymb,amscd,amsthm}
\usepackage{subfigure}
\usepackage{latexsym}
\usepackage{colortbl}
\usepackage{xcolor}
\usepackage{ae}
\usepackage{enumitem}
 \usepackage{soul}
   \usepackage{xspace}

\usepackage{tikz-cd}

\newtheorem{definition}{Definition}[section]
\newtheorem{lemma}[definition]{Lemma}

\newtheorem{corollary}[definition]{Corollary}
\newtheorem{proposition}[definition]{Proposition}
\newtheorem{theorem}[definition]{Theorem}
\newtheorem{fact}[definition]{Fact}
\newtheorem{remark}[definition]{Remark}
\newtheorem{example}[definition]{Example}

\def\N{{\mathbb N}}
\def\R{{\mathbb R}}

\def\T{{\operatorname{T}}}

\def\Hess{\mathrm{Hess}}

\def\mean{\mathbf{Mean}}
\def\meanm{\mathbf{MM}}
\def\CH{\mathrm{CH}}

\def\vol{\mathbf{vol}}

\def\crit{\mathrm{mean}}

\def\minmax{mean\xspace}

\def\H{{\mathbb H}}

\DeclareFontFamily{OT1}{pzc}{}
\DeclareFontShape{OT1}{pzc}{m}{it}{<-> s * [0.99] pzcmi7t}{}
\DeclareMathAlphabet{\mathscr}{OT1}{pzc}{m}{it}
\newcommand{\cM}{\mathscr{co}\mathcal{M}\!\mathscr{in}}

\newcommand{\co}{^*\mathbb{R}^{d,1}}

\newcommand{\mess}{\mathbf{Mess}}
\newcommand{\earth}{\mathbf{InfEarth}}
\newcommand{\gold}{\mathbf{Gold}}
\newcommand{\teich}{\operatorname{Teich}_S}
\newcommand{\length}{\mathbf{length}}
\newcommand{\infdef}{\mathbf{InfDef}}
\newcommand{\shape}{\mathbf{shape}}

\newcommand{\cod}{\mathcal{C}\hspace{-0.06 cm}od^\Gamma}
\newcommand{\Cod}{\mathbf{Cod}}

\newcommand{\hess}{\mathrm{Hess}\;}
\newcommand{\grad}{\mathrm{grad}}
\newcommand{\II}{\textsc{I\hspace{-0.05 cm}I}}

\newcommand{\gh}{g_{\mathbb{H}^ d}}

\newcommand{\gco}{g_{\mathscr{co}\mathcal{M}^{d+1}}}

\usepackage{caption}
\providecommand{\keywords}[1]{\textbf{\textit{Keywords---}} #1}

\title{Quasi-Fuchsian co-Minkowski manifolds}
\author{Thierry Barbot and Fran\c{c}ois Fillastre}

\begin{document}

\maketitle

 \begin{abstract}
This survey is an introduction to the geometry of co-Minkowksi space, the space of unoriented spacelike hyperplanes of the Minkowski space. Affine deformations of  cocompact lattices of hyperbolic isometries act on it, in a way similar to the way  that quasi-Fuchsian groups act on hyperbolic space. In particular, there is a convex core. There is also  a unique ``mean'' hypersurface, i.e. with traceless second fundamental form. The mean distance between the mean hypersurface and the lower boundary of the convex core endows the space of affine deformations of a given lattice with an asymmetric norm. The symmetrization of the asymmetric norm is simply the volume of the convex core.

In dimension $2+1$, the asymmetric norm is the total length of the bending lamination of the lower boundary component of the convex core. We obtain an extrinsic proof  of a theorem of Thurston saying that, on the tangent space of Teichm\"{u}ller space, the total length of measured geodesic laminations is an asymmetric norm.

We also exhibit and comment the Anosov-like character of these deformations, similar to
the Anosov character of the quasi-Fuchsians representations pointed out in \cite{anosovdomain}.

 \end{abstract}


\tableofcontents

\section{Introduction}

\paragraph{Action of hyperbolic isometries on model spaces}

Let $\H^d/\Gamma$ be an oriented compact hyperbolic manifold. In the Klein projective model, the hyperbolic space $\H^{d+1}$ is the interior of a ball, and some features of the action of $\Gamma$  can be described looking at the exterior of the ball, naturally endowed with a Lorentzian structure of constant curvature one, and called de Sitter space. Using affine duality with respect to the unit sphere, de Sitter space can be seen as the space of totally geodesic hypersurfaces of $\H^d$.

Since the work of G. Mess \cite{Mes07,Mes07+}, the action of cocompact lattices of $O(d,1)$ on Lorentzian constant curvature model spaces attracted attention from geometers, see e.g. the surveys \cite{fil-smi,bar-sur}. Apart from de Sitter space, Anti-de Sitter space has constant curvature $-1$ and Minkowski space is the flat one. As we said, de Sitter space is the dual of the hyperbolic space, and Anti-de Sitter space is its own dual, see e.g. \cite{FS-hyp}. Co-Minkowski space is the dual of Minkowski space. More precisely, it is the space of spacelike hyperplanes of Minkowski space. It comes with a degenerate metric of constant curvature one.

In other terms, if one wants to look at the action of subgroups of $O(d,1)$ on $d+1$ dimensional model spaces\footnote{We call a $d$-dimensional \emph{model space} the quotient by the antipodal map of a pseudo-sphere in $\R^{d+1}$, see \cite{FS-hyp}.}, up to duality, it is the same to consider
Lorentzian model spaces or constant curvature $-1$  model spaces:
\begin{center}
{\renewcommand{\arraystretch}{2}
\begin{tabular}{c c c}
  \hline
  Curvature $-1$ spaces & $\overset{dual}{\longleftrightarrow}$ & Lorentzian spaces  \\   \hline
 Hyperbolic space & $\longleftrightarrow$ & de Sitter space \\
 co-Minkowski space & $\longleftrightarrow$ & Minkowski space  \\
 Anti-de Sitter space &$\longleftrightarrow$ & Anti-de Sitter space \\
\end{tabular}}
\end{center}

\begin{figure}
\begin{center}
\psfrag{H}{Hyperbolic space}
\psfrag{C}{Co-Minkowski space}
\psfrag{A}{Anti-de Sitter space}
\psfrag{R}{Riemannian}
\psfrag{D}{Degenerate}
\psfrag{L}{Lorentzian}
\includegraphics[width=0.9\hsize]{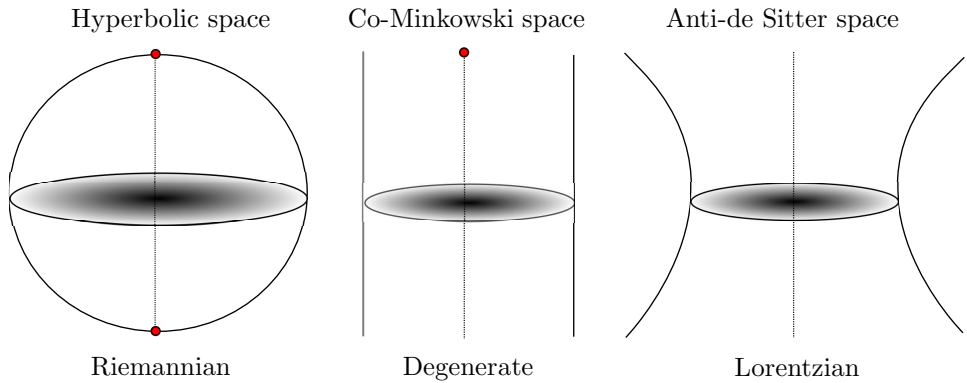}
\end{center}\caption{Affine models of the three $3$d model spaces of constant curvature $-1$. Shadowed discs are totally geodesics embedded hyperbolic planes.}\label{fig:models}
\end{figure}

\paragraph{Co-Minkowski space}

The first part of this survey is an elementary introduction to co-Minkowski space $^*\R^{d,1}$. This space has recently attracted attention under the name "half-pipe", as introduced by J. Dancinger in \cite{dan-thesis,dan}\footnote{The surface  $\cM^{1+1}$ in Figure~\ref{fig:co tilde} would deserve the name \emph{half-pipe}.
The name co-Minkowki space comes from the particular situation of this \emph{co-pseudo-Euclidean space}, see the corresponding entry in the \emph{Encyclopædia of Mathematics}.}, and used in recent works \cite{andreathesis,sch-survey,DMS},
see also \cite{FS-hyp}.

We will focus on a ``Klein model'' of co-Minkowski space as the subspace $B^d\times \R$ of the affine space $\R^{d+1}$, where $B^d$ is an open unit ball, see Figure~\ref{fig:models}.
	In general, the interest of an affine model is that (unparameterized)  geodesics are affine segments, so for example some affine notions as convexity or convex hull are easily tractable. In the particular case of
	co-Minkowski space, many analogues of  classical differential geometry results are easier than the original ones, for example:
	\begin{itemize}
	\item the (smooth) hypersurfaces carrying a non-degenerate induced metric are all hyperbolic, and when they are metrically complete, they are    graphs of functions on the ball $B^d$,
	\item the shape operator of graph hypersurfaces gives symmetric Codazzi tensors on the hyperbolic space $\H^d$,
	\item actually, the correspondence between complete hyperbolic hypersurfaces and hyperbolic symmetric Codazzi tensor is one-to-one, that gives a simplified co-Minkowski version of the fundamental theorem of hypersurfaces (Section~\ref{sec: extr}),
\item complete hyperbolic hypersurfaces such that the trace of the shape operator vanishes are called \emph{\minmax surfaces}; existence and uniqueness of such hypersurfaces are straightforward consequence of classical theory of elliptic PDE on the ball (Section~\ref{sec minmax}),
\item the functions whose graph is  a boundary of the convex hull of the graph of a continuous map $b:\partial B^d\to \R$  are solutions of the classical
Monge--Amp\`ere equation (Section~\ref{sec convex hull}).
\end{itemize}

Another nice feature of the cylinder model of co-Minkowski space is that it allows an easy definition of
  \emph{degenerations} of hyperbolic or Anti-de Sitter manifolds to a co-Minkowski manifold,
as Figure~\ref{fig:models} heuristically suggests. In turn,
 co-Minkowski geometry as a \emph{transitional geometry} between the hyperbolic geometry and the AdS geometry was the main motivation of \cite{dan,dan-thesis}, see also \cite{andreathesis,FS-hyp}. Unfortunately, such considerations are out of the scope of the present survey.

\paragraph{The action of $H^1(\Gamma,\R^{d,1})$}

By duality, the group of isometries of Minkowski space, that is $O(d,1) \ltimes \R^{d,1}$, acts on co-Minkowski space, preserving the degenerate metric (see Remark~\ref{iso group larger}). For our purpose, it will be more relevant to restrict ourselves to the action of
$O_0(d,1) \ltimes \R^{d,1}$, where
$O_0(d,1)$ is the connected component of  the identity of $O(d,1)$. If $\Gamma$ is a Kleinian cocompact subgroup of $O_0(d,1)$, then the representations of $\Gamma$ into
$O_0(d,1) \ltimes \R^{d,1}$
 are parametrized by maps $\tau:\Gamma\to \R^{d,1}$ satisfying a cocycle relation. Let $Z^1(\Gamma,\R^{d,1})$ be the space of cocycles.

The choice of two totally geodesic embedding of $\H^d$ (on which $\Gamma$ acts) into co-Minkowski space will give different cocycles, related by a coboundary conditions.
So  we are interested in the space
 $H^1(\Gamma,\R^{d,1})$, the quotient of the space of cocycles by the coboundaries. From an extrinsic point of view, the vector space  $H^1(\Gamma,\R^{d,1})$ is the space of deformations of $\Gamma$ into the group of affine
isometries, up to conjugacy by translations.
But
$H^1(\Gamma,\R^{d,1})$ encodes
many much informations:
\begin{itemize}
\item  due to Mostow rigidity theorem, for $d>2$, it is not possible to non-trivially deform $\Gamma$ among Kleinian subgroups of $O(d,1)$. But it is possible to look at deformations of the canonical representation of $\Gamma$ into $O(d+1,1)$, that corresponds to the deformation of the flat conformal structure of $\H^d/\Gamma$.  At an infinitesimal level,
the deformations are parametrized by
$H^1(\Gamma,\mathfrak{so}(d+1,1))$.
Due to the well-known splitting $\mathfrak{so}(d+1,1)=\mathfrak{so}(d,1)\oplus \R^{d,1}$, we have that
 $$H^1(\Gamma,\mathfrak{so}(d+1,1))=H^1(\Gamma,\mathfrak{so}(d,1))\oplus H^1(\Gamma,\R^{d,1})$$
but due to the Calabi--Weil infinitesimal rigidity theorem  $H^1(\Gamma,\mathfrak{so}(d,1))$ reduces to $0$ \cite[8.10]{kapovich}.

On the other hand, $H^1(\Gamma,\R^{2,1})$ is also isomorphic, as a linear space, to the tangent
space of the Teichm\"uller space at (the conjugacy class of) $\Gamma$, when we consider the Teichm\"uller space
as the space of discrete, faithful representations of $\Gamma$ into the isometries of the hyperbolic plane up to conjugacy,  see Section~\ref{cas 2+1};
\item there is a natural
isomorphism between $H^1(\Gamma,\R^{d,1})$ and the space of traceless symmetric Codazzi tensors on $\H^d/\Gamma$ (see Proposition~\ref{prop:ident H1 cod} for a proof using extrinsic co-Minkowski geometry), and the space of  traceless symmetric Codazzi tensors parametrizes the space of infinitesimal  deformations of the flat conformal structure of $\H^d/\Gamma$, as well as the space of infinitesimal deformations of the Riemannian metric of $\H^d/\Gamma$ preserving the total volume and the harmonicity of the curvature \cite{lafontaine};
\item $H^1(\Gamma,\R^{d,1})$ parametrizes the space of future complete flat globally hyperbolic maximal Cauchy compact spacetimes (in short, future complete flat  GHMC spacetimes), with $\Gamma$ as the linear part of the holonomy, see \cite{Mes07,Mes07+,Bar05,Bon05} for more details and precise definitions. The universal covers of such spacetimes isometrically embed as convex sets in Minkowski space, whose duals in co-Minkowski space define the convex cores that will be mentioned below, see Remark~\ref{rem flat ghmc}.
\end{itemize}

As a consequence of the first point, for $d=2$, $H^1(\Gamma,\R^{d,1})$ is a vector space of dimension $(6g-6)$, where $g$ is the genus of $\H^2/\Gamma$. For $d>3$, it is not clear whether $H^1(\Gamma,\R^{d,1})$  is trivial or not. A classical result is that it has dimension at least $r$ if $\H^d/\Gamma$ contains $r$ disjoint embedded totally geodesic hypersurfaces \cite{lafontaine,kourou,JM}. We give an elementary co-Minkowski proof of this fact in Section~\ref{space of cocycle}.
See for example \cite{stamping} and \cite{JM} for more informations, and  \cite{BS} for up-to-date references about this question.

The action of  $\Gamma_\tau$, that is $\Gamma$ deformed by an element $\tau$ of $Z^1(\Gamma,\R^{d,1})$, onto co-Minkowski space is also interesting in its own. Namely, here too, it is a baby toy model, this time comparing to the study of quasi-Fuchsian hyperbolic manifolds on the one hand, and to AdS GHMC manifolds on the other one
(they are the Lorentzian analogues of  quasi-Fuchsian hyperbolic manifolds). We will focus on the following aspects. Let $\tau \in Z^{1}(\Gamma,\R^{d,1})$.
\begin{itemize}
\item  There exists a smooth hypersurface invariant under the action of $\Gamma_\tau$. This is a simple illustration of the general  ``Ehresmann--Weil--Thurston principle'', see Proposition~\ref{prop exist smooth cvx}.
\item The group $\Gamma_\tau$ acts freely and properly discontinuously on co-Minkowski space, and the quotient gives a $(d+1)$-dimensional manifold homeomorphic to $\H^d/\Gamma\times \R$ (see Lemma~\ref{lem:action propre}).
\item The co-Minkowski manifold $^*\R^{d,1}/\Gamma_\tau$
has a \emph{convex core}, i.e. it contains a non-empty compact convex set. So the action of $\Gamma_\tau$ on co-Minkowski space is \emph{convex cocompact} in the sense of \cite{DGK1,DGK2}.
\item The co-Minkowski manifold $^*\R^{d,1}/\Gamma_\tau$
contains a unique ``\minmax'' hypersurface, that is with vanishing mean curvature. This situation is reminiscent of  \emph{almost Fuchsian manifolds}, a particular case of quasi-Fuchsian manifolds which contain a unique minimal surface, see \cite{KS}.
\item Moreover, $^*\R^{d,1}/\Gamma_\tau$ is foliated by CMC hypersurfaces, equidistant to the \minmax hypersurface, see Remark~\ref{rem: foliation CMC}.
\end{itemize}

We consider that  co-Minkowski space is a toy model, because with a pedestrian approach, we are able to give an almost self-contained presentation of the different properties evoked above.

\paragraph{An asymmetric norm}

Until this point, all the mentioned results were previously more or less known, at least under the form of  dual statements in Minkowski space.
Also, the present survey  contains the following original contribution.

As we said, the quotient of co-Minkowski space by $\Gamma_\tau$ has a convex core, and a unique \minmax hypersurface,  contained in the convex core.
The mean distance between the lower boundary component of the convex core and this \minmax hypersurface gives a non-negative number, which is uniquely defined by the class in  $H^1(\Gamma,\R^{d,1})$ of $\tau$.
This gives a map from $H^1(\Gamma,\R^{d,1})$ to $\R_+$, which is actually an asymmetric norm on $H^1(\Gamma,\R^{d,1})$, see Section~\ref{sec asy}. We will call it the \emph{$S_1$ norm} (see Remark~\ref{remS1} for the signification of $S_1$).

The symmetrization
of the $S_1$ norm is:
\begin{itemize}
\item the volume of the convex core;\footnote{This fact was noted to the first author by Andrea Seppi.}
\item a ``mean distance'' between the future complete and the past complete flat GHMC having the same holonomy  (see
Remark~\ref{rem flat ghmc}).

\end{itemize}

In dimension $2$, it appears that this asymmetric norm corresponds to the \emph{earthquake norm} introduced by Thurston in \cite{thu98}. In particular, we obtain a new proof of Theorem~5.2 in \cite{thu98}, saying that the earthquake norm is an asymmetric norm on the tangent of Teichm\"uller space. The tangent space of Teichm\"uller space can be identified with the space of measured geodesic laminations, and the earthquake norm in the total length of the lamination, see Section~\ref{cas 2+1}.

In turn, the volume of the convex core is the sum of the total length of the bending laminations of its boundary. Here again, this result should be compared with its more involved analogues in the hyperbolic and anti-de Sitter cases \cite{brock,BST}.

Using two successive identifications of the tangent space of Teichm\"uller space with its cotangent space and a formula of Wolpert, the earthquake norm defines another asymmetric norm on the tangent space of Teichm\"uller space, the \emph{length norm}, see \eqref{eq length norm} for a formula. The length norm
defines an asymmetric Finsler structure on Teichm\"uller space, that in turn defines a distance, now called the \emph{Thurston asymmetric distance}, and introduced by Thurston in \cite{thu98}. This distance recently attracted attention \cite{PT,PS,walsh}.
Note that the earthquake norm also induces an asymmetric distance on Teichm\"uller space, but, to the best of our knowledge, nothing is known about this distance.

\paragraph{Anosov feature}

In the third and last part of the present survey, we see that co-Minkowski space is also a baby toy model for the
theory of Anosov representations, which has known during the recent years, after the pioneering work of F. Labourie \cite{labourie} a series of development (see \cite{anosovproper}, \cite{amalgam}, \cite{anosovdomain}, \cite{pressure}, \cite{kapovich2}, see also \cite{bar-sur} for a complementary discussion on
Anosov representations in the context of Lorentzian geometry, and \cite{Ghosh} for a proof of the Anosov
character of the representations considered in the present survey).

Once more, it turns out that in the context of co-Minkowski space the theory of Anosov representations
reduces to a particularly simple form. Moreover, this point of view provides a proof of the fact that convergence of cocycle implies \emph{uniform} convergence of limit curves (Lemma \ref{le:tauuniform}).

\paragraph{Acknowledgement}
The authors would  like to thank the organizers of the conference ``Moduli spaces and applications in geometry, topology, analysis and mathematical physics'' in Beijing to offer them the opportunity of writing the present paper.

The present paper is also part of the Math Amsud 2014 project n°38888QB-GDAR.

\section{Co-Minkowski geometry}

Co-Minkowski space is the space of (unoriented) spacelike hyperplanes of Minkowski space. We first investigate the space of oriented spacelike hyperplanes  (Section~\ref{sec:co}).
Then we introduce a cylindrical  affine model for co-Minkowski space, similar to the Klein ball model of hyperbolic space (Section~\ref{sec: cyl}). In the cylindrical model,
the co-Minkowski space is the cylinder $B^d\times \R$, where $B^d$ is the open unit ball of $\R^d$ centered at the origin.  In particular, extrinsic co-Minkowski geometry of graphs of maps $h:B^d\to \R$
can be investigated (Section~\ref{sec:extrin}).

\subsection{Definition of co-Minkowski space}\label{sec:co}

\subsubsection{Space of spacelike hyperplanes}

Let us recall that the \emph{Minkowski space} $\R^{d,1}$ of Lorentzian geometry is the affine space $\R^{d+1}$ endowed with the bilinear form
$$\langle x,y\rangle_{d,1}=x_1y_1+\cdots+x_dy_d-x_{d+1}y_{d+1}~. $$
 A hyperplane $P$ of  $\R^{d,1}$ is \emph{spacelike} (resp. \emph{timelike}, \emph{lightlike}) if the restriction of $\langle \cdot,\cdot\rangle_{d,1}$ to $P$ is positive-definite (resp. has signature $(+,\ldots,+,-)$, is degenerate).
The isometry group of $\R^{d,1}$ is $O(d,1) \ltimes \R^{d,1}$: it is made of translations and linear transformations preserving  $\langle \cdot,\cdot\rangle_{d,1}$.

Linear spacelike hyperplanes are parametrized by the set of \emph{future} unit normal vectors (for $\langle \cdot,\cdot\rangle_{d,1}$):
$$\mathcal{H}^d :=\{x\in \R^{d,1} | \langle x,x\rangle_{d,1}=-1, x_{d+1}>0 \}~. $$
Let $g_{\mathcal{H}^d}$ be the metric induced by
$\langle \cdot,\cdot\rangle_{d,1}$ on the tangent spaces of
$\mathcal{H}^d$.
It is well known that $(\mathcal{H}^d,g_{\mathcal{H}^d})$
is a model of the $d$-dimensional hyperbolic space.

Let $P$ be an affine spacelike  hyperplane of $\R^{d,1}$. If $n\in \mathcal{H}^d\subset \R^{d,1}$ is the  future  timelike unit normal to $P$, then there exists $h\in \R$ such that
$$P=\{y\in \R^{d,1} | \langle y,n\rangle_{d,1} = h \}~. $$
This defines a point $$\tilde{P}^*=(n,h)$$ in $\R^{d+1}\times \R=\R^{d+2}$. More precisely, the point $\tilde{P}^*$ belongs to one of the connected component of the degenerate quadric
$$\cM^{d+1}:=\{x\in \R^{d+2} | \langle x,x\rangle_{d,1,0}=-1 \}
$$
where
$$\langle (x_1,\ldots,x_{d+1},x_{t}),(y_1,\ldots,y_{d+1},y_{t})\rangle_{d,1,0}=x_1y_1+\cdots+x_{d}y_{d}-x_{d+1}y_{d+1}~. $$
Note that $\cM^{d+1}$ is the space of oriented spacelike hyperplanes of Minkowski space. See Figure~\ref{fig:co tilde}.

We will denote by $\gco$ the degenerate $(0,2)$-tensor induced by  $\langle \cdot,\cdot\rangle_{d,1,0}$ on the tangents
spaces of $\cM^{d+1}$. The connected component
$\cM^{d+1}_+=\cM^{d+1}\cap \{x_{d+1}>0 \}$ of $\cM^{d+1}$ containing the point $\tilde P^*$
is homeomorphic to $\mathcal{H}^d\times \R$. In those coordinates,
the degenerate metric $\gco$ on $\cM^{d+1}_+$ writes as
$$\gco=g_{\mathcal{H}^d} + 0 \mbox{d}x_{t}~. $$

\begin{figure}
\begin{center}
\psfrag{mink}{$\R^{d,1}$}
\psfrag{comink}{ $\cM^{d+1}$}
\psfrag{v}{$n$}
\psfrag{p}{$P$}
\psfrag{b}{$B^1$}
\psfrag{xu}{$\binom{x}{1}$}
\psfrag{x}{$x$}
\psfrag{hx}{$h_x$}
\psfrag{ps}{$P^*$}
\psfrag{hv}{$h$}
\psfrag{tps}{$\tilde{P}^*$}
\psfrag{c}{$\cM^{d+1}_+$}
\psfrag{cc}{$\co$}
\psfrag{H}{$\mathcal{H}^d$}
\psfrag{xd1}{$x_{d+1}$}
\psfrag{xd2}{$x_{t}$}
\psfrag{x1}{$x_1$}
\psfrag{o}{$0$}
\includegraphics[width=0.8\hsize]{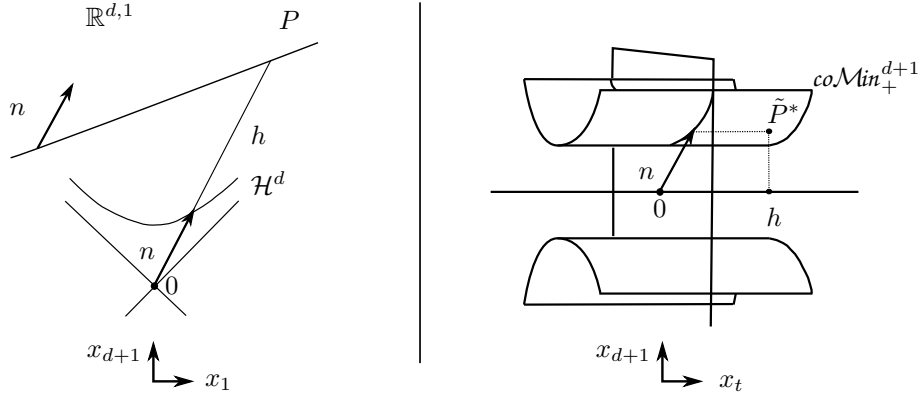}
\end{center}\caption{The dual of a spacelike hyperplane of Minkowski space $\R^{d,1}$ in $\cM^{d+1}$. In the picture, $d=1$.}\label{fig:co tilde}
\end{figure}

We also introduce the fibration:
$$\mathfrak{\pi}: \cM^{d+1}_+ \to \mathcal{H}^d$$
mapping $(x_1,\ldots, x_{d+1}, x_{t})$ to $(x_1,\ldots,x_{d+1})$. It is a principal $\mathbb{R}$-bundle;
it is an isometry, and the fibers are precisely tangent to the kernel of the degenerate metric $\gco$.

\subsubsection{Isometries}

As the ``metric'' $\gco$ is degenerate, it will be more relevant to consider a group acting on $\cM^{d+1}$.
 As an isometry of Minkowski space sends spacelike hyperplanes onto spacelike hyperplanes, it acts naturally on $\cM^{d+1}$. This is the way we define the isometry group of $\cM^{d+1}$.
More precisely,  it is immediate that if
$$P=\{y\in \R^{d,1} | \langle y,n\rangle_{d,1} = h \} $$
is a spacelike hyperplane of $\R^{d,1}$ and $A\in O(d,1)$, so that
$\tilde{P}^*=(n,h)$,  then
$$\widetilde{AP}^*=(An,h)$$ and if $v\in \R^{d,1}$,
$$\widetilde{P+v}^*=(n,\langle v,n\rangle_{d,1}+h)~.$$

So  $ \mbox{O}(d,1)\ltimes \R^{d,1}$  acts linearly on $\cM^{d+1}$
 via the representation
\begin{equation}\label{eq:isom co}(A,v)\mapsto
\left(
\begin{array}{ccc|c}

  & & & 0 \\
   & \raisebox{-1pt}{{\huge\mbox{{$A$}}}}  & & \vdots  \\
  & & & 0 \\ \hline
   & ^tvJA &  & 1
\end{array}
\right)~, \end{equation}
where $J=\operatorname{diag}(1,\ldots,1,-1)$ (recall that $A\in O(d,1)$ if and only if $^tA=JA^{-1}J$).
So we define the \emph{isometry group} of $\cM^{d+1}$ as
 $ O(d,1)\ltimes \R^{d,1}$ with the action on $\R^{d+2}$ induced by
the representation \eqref{eq:isom co}.
In particular, the group structure on
 $ O(d,1)\ltimes \R^{d,1}$ is
 \begin{equation}\label{group structure}
 (A_1,v_1)\cdot (A_2,v_2)=(A_1A_2,v_1+A_1v_2)~.
 \end{equation}

\begin{remark} \label{rk:affinefiber}{\rm
		
		Let $O_+(d,1)$ be the subgroup of $O(d,1)$ preserving $\mathcal H^{d}$. Then, $O_+(d,1) \ltimes \R^{d,1}$
		preserves the connected component $\cM^{d+1}_+$, and the fibration $\mathfrak{\pi}$ is $O_+(d,1) \ltimes \R^{d,1}$-equivariant. The elements of $O_+(d,1) \ltimes \R^{d,1}$ inducing the identity map on $\mathcal H^d$
		are precisely the translations (elements of $\R^{d,1}$).
		
		Every fiber of $\mathfrak{\pi}$ admits a natural affine structure, for which they are individually isomorphic
		the real line. The action of $O_+(d,1) \ltimes \R^{d,1}$ preserves this affine structure along the fibers.
	}\end{remark}

\begin{remark}\label{iso group larger}{\rm

The isometry group of $\cM^{d+1}$ is smaller than the group of
transformations preserving the degenerate metric $\gco$. For example, for $c>0$, the map $H_c:\R^{d+2}\to \R^{d+2}$,
$H_c(x)=(x_1,x_2,\ldots,x_{d+1},cx_{t})$, preserves $\langle \cdot,\cdot\rangle_{d,1,0}$ (hence it preserves $\cM^{d+1}$ and $\gco$), but by definition it is not an isometry of $\cM^{d+1}$.

There does not exist any (non-degenerate) semi-Riemannian metric on $\cM^{d+1}$ invariant under the isometry group of $\cM^{d+1}$ \cite[Fact~2.27]{FS-hyp}.
}\end{remark}

\subsubsection{Connection, geodesics}

We have now the hypersurface $\cM^{d+1}$ in $\R^{d+2}$
together with an ``isometry group" and a degenerate metric $\gco$. As those elements are coming from the  degenerate form
$\langle \cdot,\cdot\rangle_{d,1,0}$ on the ambient $\R^{d+2}$, there is no obvious metric notion of
``unit normal vector'' to $\cM^{d+1}$. Nevertheless, we can proceed similarly to classical affine differential geometry \cite{nomizu}.
Namely, at a point $x\in \cM^{d+1}$, let us define as a ``normal field'' the vector field $\operatorname{N}(x)=x$. Obviously, $\operatorname{N}$ is  transverse to  $\cM^{d+1}$ and invariant under the group of isometries of $\cM^{d+1}$.
 The choice of this normal field allows to define a connection $\nabla^{\mathscr{co}\mathcal{M}^{d+1}}$ on $\cM^{d+1}$ induced by the canonical connection $D$ of the ambient linear space $\R^{d+2}$:
$$D_YX=\nabla^{\mathscr{co}\mathcal{M}^{d+1}}_YX + \langle X,Y\rangle_{d,1,0}\operatorname{N}~. $$

The following facts are easily checked, see  \cite[Section~4.2]{FS-hyp}.

\begin{fact}\label{fact geod com tilde}
The connection $\nabla^{\mathscr{co}\mathcal{M}^{d+1}}$ has the following properties:
\begin{itemize}[nolistsep]
\item it is torsion free,
\item compatible with the degenerate metric $\gco$,
\item invariant under isometries,
\item its (unparameterized) geodesics are intersection of $\cM^{d+1}$ with linear planes of $\R^{d+2}$.
\end{itemize}
\end{fact}

It follows from the last point that the intersection of
 $\cM^{d+1}$ with linear $k$-planes of $\R^{d+2}$ are totally geodesic. Those intersections will play a fundamental role as the following fact shows.

 \begin{fact}\label{fact hyp k tilde}
 The intersection of
 $\cM^{d+1}$ with a linear $k$-planes of $\R^{d+2}$ is isometric
 (for the metric induced by $\gco$) to the hyperbolic space of dimension $k$.

Moreover, $\nabla^{\mathscr{co}\mathcal{M}^{d+1}}$ coincides with the Levi-Civita connection of the hyperbolic metric on any such subspace.
 \end{fact}
\begin{proof}
Immediate as one can always find an isometry of  $\cM^{d+1}$ sending a linear $k$-plane to a linear $k$-plane contained in $\{x_{t}=0\}$.
\end{proof}

\subsubsection{Co-Minkowski space}

The co-Minkowski space is the space of unoriented spacelike hyperplanes of Minkowski space, that is the quotient of $\cM^{d+1}$ by the antipodal map.

\begin{definition}
The  \emph{co-Minkowski space} $\co$ is the following subspace of the projective space:
$\co=\cM^{d+1}/\{\pm \operatorname{Id}\}$, endowed with the push-forward of the degenerate metric $\gco$, denoted by $g_{\co}$.
\end{definition}

The connection  $\nabla^{\mathscr{co}\mathcal{M}^{d+1}}$
also induces a connection $\nabla^{\co}$ on $\co$.

We define the \emph{isometry group} of $\co$ as the image of $ \mbox{O}(d,1)\ltimes \R^{d,1}$ into $\mbox{PGL}(d+2)$, by a projective quotient of   the representation
given by \eqref{eq:isom co}.

The map $\mathfrak{\pi}: \cM^{d+1}_+ \to \mathcal{H}^d$ induces a $\R$-fibration $^*\pi:  {\co} \to \mathcal H^d$, which is an isometry, and $ \mbox{O}(d,1)\ltimes \R^{d,1}$-equivariant.

In will be interesting to work in a particular affine model of co-Minkowski space. This will be the cylindrical coordinates introduced in the next section.

\subsection{Cylindrical model}\label{sec: cyl}

\subsubsection{Klein ball model of the hyperbolic space}\label{sec klein}

We have seen that the subspace $\mathcal{H}^d$ of Minkowski space,
endowed with the induced metric, is a model of the hyperbolic space.
It is isometric to the
subset  $\{x\in \R^{d,1}| \langle x,x\rangle_{d,1} <0 \}$ of the projective space $\mathbb{P}(\R^{d,1})$ endowed with the push-forward metric.

The \emph{Klein ball model} of the hyperbolic space
is the image of the projective model
of the hyperbolic space
in the affine chart $\{x_{d+1}=1\}$. As a set, it is the open Euclidean unit ball $B^d$. The push-forward of the hyperbolic metric on $B^d$ is denoted by $\gh$. We will sometimes  use the notation $\H^d$ to designate the hyperbolic space $(B^ d,\gh)$.
In the remainder of this section, we give explicit formulas relating the hyperbolic geometry on $B^d$ to the standard Euclidean geometry on $B^d$, that will be needed in the sequel of the paper.

If $x\in B^d$, then the vector  $\binom{x}{1}$ of $\R^{d,1}$ rescaled by the factor $L^{-1}(x)$ belongs to $\mathcal{H}^d$, where
$$L(x)=\sqrt{1-\|x\|^2} $$
and $\|\cdot\|$ is the Euclidean norm on $B^d$:
$$\|(x_1,\ldots,x_d)\|=\sqrt{x_1^2+\cdots+x_d^2}~,$$
see Figure~\ref{fig:hyp}.

\begin{figure}
\begin{center}
\psfrag{H}{$\mathcal{H}^d$}
\psfrag{x1}{$\binom{x}{1}$}
\psfrag{lx1}{$L^{-1}(x)\binom{x}{1}$}
\psfrag{B}{$B^d$}
\psfrag{o}{$0$}

\includegraphics[width=0.4\hsize]{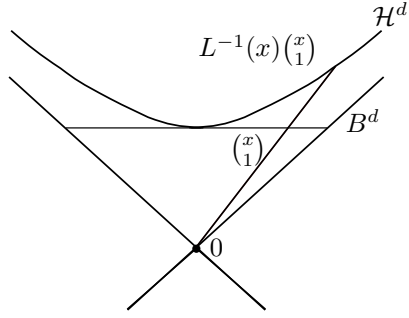}
\end{center}\caption{The hyperboloid $\mathcal{H}^d$ and the Klein ball model of the hyperbolic space.}\label{fig:hyp}
\end{figure}

The expression of the hyperbolic metric $\gh$ in the Klein ball model is:
\begin{equation}\label{met hyp}\gh(x)(X,Y)=L(x)^{-2}\langle X,Y\rangle_d + L(x)^{-4}\langle x,X \rangle_d\langle x,Y\rangle_d~.\end{equation}
where $\langle \cdot,\cdot\rangle_d$ is the standard Euclidean metric
on $\R^d\supset B^d$, $x\in B^d$, $X,Y\in T_xB^d\cong \R^d$.
In order to help computations, one may note that
\begin{equation}\label{eq der L-1}
D_XL^{-1}(x)=L^{-3}(x)\langle x,X\rangle_d
\end{equation}
and
\begin{equation}\label{eq hess L}\operatorname{Hess} L=-L\gh~,\end{equation}
where $\operatorname{Hess}$ is the usual Hessian on $\R^d$.

 If $\omega_{B^d}$ if the restriction to $B^d$ of the Euclidean volume form, and $\omega_{\H^d}$ is the  volume form on $B^d$ associated to the hyperbolic metric $g_{\H^d}$, from \eqref{met hyp} one obtains

\begin{equation}\label{eq:volume}\omega_{B^d}=L^{d+1}\omega_{\H^d}~. \end{equation}

The main feature of the Klein ball model of the hyperbolic space is that the (unparameterized) geodesics of $\gh$ are exactly the affine segments in $B^d$. This is straightforward, as the geodesics of $\mathcal{H}^d$ are the intersections of
$\mathcal{H}^d$  with linear timelike planes of $\R^{d,1}$.
This gives the following correspondence between the connections, see \cite[Lemma 4.17]{FS-hyp}.
\begin{proposition}[Weyl formula]
If $\nabla^{\H^d}$ is the Levi-Civita  connection of $\gh$
and $D$ is the canonical connection  on $B^d$, then
\begin{equation}\label{eq con}\nabla^{\H^d}_X Y = D_XY + L^{-2}(x)(\langle x,X\rangle_d Y + \langle x,Y\rangle_d X)~. \end{equation}
\end{proposition}

\begin{corollary}
If $\operatorname{Hess}^{\H^2}$ is the Hessian given by $\nabla^{\H^2}$,  then, for a smooth map $f:B^d\to \R$,
\begin{equation}\label{nabla2}\operatorname{Hess}^{\H^2}f(x)(X,Y)=\operatorname{Hess}f(x)(X,Y) - L^{-2}(x)(\langle x,X\rangle_d  \operatorname{d}f(x)(Y) + \langle x,Y\rangle_d  \operatorname{d}f(x)(Y))~. \end{equation}

Also,
\begin{equation}\label{eq hyp hess}L^{-1}(x)\hess f(x)(X,Y)=\left( \operatorname{Hess}^{\H^2}(L^{-1}f)(x)(X,Y) - (L^{-1}f)(x) \gh(x)(X,Y)\right)~.\end{equation}

\end{corollary}

\begin{proof}
 \eqref{nabla2} follows from \eqref{eq con} and
\begin{equation}\label{hess con}\operatorname{Hess}^{\H^2}f(x)(X,Y)=X.Y.f(x) - \operatorname{d}f(x)(\nabla^{\H^d}_XY)~.\end{equation}

Finally, \eqref{eq hyp hess} comes from \eqref{nabla2}, \eqref{met hyp} and
\begin{equation}\label{eq: hessien produit}\hess f g = f \hess g + g\hess  f +  \operatorname{d}f \otimes  \operatorname{d}g +  \operatorname{d}g \otimes  \operatorname{d}f~. \end{equation}
\end{proof}

\begin{fact}\label{facts laplacian}
If $\Delta$ is the Euclidean Laplacian on $B^d$, then
\begin{equation}\label{trace hessien}\operatorname{tr}_{\gh} \operatorname{Hess} f (x) = L^2(x)( \Delta f - \operatorname{Hess} f(x)(x,x))~. \end{equation}

If $\Delta^{\H^d}$ is the Laplacian on $B^d$ given by $g_{\H^d}$, then
\begin{equation}\label{eq delta}\operatorname{tr}_{\gh} L^{-1}\hess f=\Delta^{\H^d} (L^{-1}f) - d (L^{-1}f)~.\end{equation}

\end{fact}
\begin{proof}
Let $A$ be the linear operator such that
$\operatorname{Hess}f(x)(X,Y)=\gh(x)(AX,Y)$.
For $x\not=0$, let $(e_i)_{1,\cdots,d}$ be an orthonormal Euclidean basis of $T_xB^d$, such that $e_1=x/\|x\|$.
The definition of $A$ and \eqref{met hyp} give, for $i>1$,
$$\langle Ae_i,e_i\rangle_d=L^2(x)g_{\H^d}(x)(Ae_i,e_i)=L^2(x)\operatorname{Hess} f(x)(e_i,e_i)~, $$
and
$$\langle A e_1,e_1\rangle_d = L^2(x)\operatorname{Hess} h(x)(e_1,e_1) + L^{-2}(x) \langle x,Ax\rangle_d~. $$
Also from the definition of $A$ and \eqref{met hyp},
$$L^{-2}(x)\langle x,Ax\rangle_d= L^2(x) g_{\H^d}(x)(x,Ax)= L^2(x) \operatorname{Hess} f(x)(x,x)~. $$
\eqref{trace hessien} follows from $\operatorname{tr}_{\gh} \operatorname{Hess} f (x)=\sum_{i=1}^d \langle Ae_i,e_i\rangle_d$.
Also, \eqref{eq delta} is immediate from \eqref{eq hyp hess}.
\end{proof}

Let us end this section with some basic facts about (smooth) hyperbolic Codazzi tensors.

\begin{definition}\label{def.hypcodazzi}
  A $(0,2)$-tensor $C$ on $\H^d$ is a \emph{(hyperbolic) Codazzi tensor}  if it satisfies the
 the Codazzi equation on $\H^d$:
$$(\nabla^{\H^d}_X)C(Y,Z)=(\nabla_Y^{\H^d})C(X,Z)~. $$
\end{definition}

\begin{lemma}
Let $C$ be a $(0,2)$-tensor on $B^d$. Then $C$ is a  hyperbolic Codazzi tensor if and only if
$$D_X(LC)(Y,Z)=D_Y(LC)(X,Z)~. $$
\end{lemma}
\begin{proof}
The definition of Codazzi tensor means that
$$X.C(Y,Z)-C(\nabla^{\H^d}_XY,Z)-C(Y,\nabla_X^{\H^d}Z)= Y.C(X,Z)-C(\nabla^{\H^d}_YX,Z)-C(X,\nabla_Y^{\H^d}Z)~.$$
Developing this expression using  \eqref{eq con}, one obtains, at a point $x$,
$$D_XC(x)(Y,Z)-L^{-2}(x)\langle x,X\rangle_dC(Y,Z)= D_YC(x)(X,Z)-L^{-2}(x)\langle x,Y\rangle_dC(X,Z)~.$$
Writing $C=L^{-1}LC$, developing the above expression and using \eqref{eq der L-1} leads to the result.
\end{proof}

\begin{fact}
Let $S$ be a  $(0,2)$-tensor on $B^d$. If $D_XS(Y,Z)=D_YS(X,Z)$, then there exists
a function $F=(F_1,\ldots,F_n)$ with $F_i:B^d\to \R$ such that $S$ is the Jacobian matrix of $F$.
\end{fact}
\begin{proof}
Let $\Omega_j=\sum_{i=1}^dS_{ij}\mbox{d}x^i$. As $\frac{\partial S_{ij}}{\partial x_k}=\frac{\partial S_{kj}}{\partial x_i}$, $\mbox{d}\Omega_j=0$, so by Poincaré Lemma, there exists a function $F_j:B^d\to \R$ such that $\mbox{d}F_j=\Omega_j$.
\end{proof}

\begin{fact}
Let $F=(F_1,\ldots,F_d)$ with $F_j:B^d\to \R$. Then there exists $f:B^d\to \R$ with $\frac{\partial f}{\partial x_i}=F_i$ if and only if $\frac{\partial F_i}{\partial x_j}=\frac{\partial F_j}{\partial x_i}$.

In other term, the Jacobian matrix of $F$ is a Hessian matrix (namely the one of $f$) if and only if it is a symmetric matrix.
\end{fact}
\begin{proof}One implication is Schwarz's theorem.  On the other direction, the one-form $\omega=\sum_{i=1}^d F_i\mbox{d}x^i$ is closed by hypothesis, hence exact by Poincaré Lemma, and it suffices to set $\omega=\mbox{d}f$. \end{proof}

We finally obtain the following classical result \cite{ferus,OS83,bsep}.

\begin{lemma}\label{lem tens cod}
Let $C$ be a  $(0,2)$-tensor on $B^d$. Then $C$ is a symmetric hyperbolic  Codazzi tensor if and only if there exists $f:B^d\to \R$ such that
$$C=L^{-1}\hess f~. $$
\end{lemma}

\subsubsection{Affine representation of co-Minkowski space}

To keep track of some relevant affine notions such as convexity,  we will work in an affine model of  co-Minkowski space.
Namely, we will consider the affine model of co-Minkowski space given by the central projection of $\cM^{d+1}_+$ onto the hyperplane $\{x_{d+1}=1\}$ of $\R^{d+2}$. Observe that in doing so, we favor the coordinate $x_{d+1}$, i.e. we distinguish the future
timelike vector $(0, \ldots, 0, 1)$ of $\R^{d,1}$. We will go back on this remark in Section \ref{sec:anosov}.
In the hyperplane $\{x_{d+1}=1\}$, the image of $\cM^{d+1}$  is
the cylinder $B^d\times \R$, where $B^d$ is the open unit ball centered at the origin of $\R^{d}$.

We denote by $\pi: B^d \times \R \to B^d$  the projection on the first factor. It corresponds
to the fibration $\mathfrak{\pi}: \cM^{d+1}_+ \to \mathcal{H}^d$.
We will call \emph{vertical lines} the fibers of $\pi$. They correspond
to parallel spacelike hyperplanes in Minkowski space.

\begin{remark}{\rm In those coordinates $B^d\times \R\subset\R^{d+1}$, the degenerate metric $g_{\co}$ of co-Minkowski space is
$g_{\H^d}+0\operatorname{d}x_{t}^2.$
The degenerate metric $g_{\co}$ defines a
''distance'' between points of co-Minkowski space.
Actually this distance is nothing but the
the Klein projective metric:
if $x,y\in  B^d\times \R$, then they are on a line meeting $\partial B^{d}\times \R \cup\{\infty\}$ either at two distinct points $I,J$, or at $I=J=\infty$. Then the Klein projective distance is  $d(x,y)=\frac{1}{2}|\ln [x,y,I,J]|$, where $[\cdot,\cdot,\cdot,\cdot]$ is the cross-ratio, see \cite{FS-hyp}.
}\end{remark}

\begin{remark}{\rm
The \emph{boundary at infinity} of co-Minkowski space is $ \partial B^d\times \R$. It parametrizes the set of lightlike affine hyperplanes of Minkowski space, and it is called
\emph{Penrose boundary} in \cite{Bar05}.
Note that $ (\R^d\setminus \bar{B}^d)\times \R$ parametrizes the set of affine timelike hyperplanes of Minkowski space, but we don't need to consider it.
}\end{remark}

\begin{figure}
\begin{center}
\psfrag{mink}{$\R^{d,1}$}
\psfrag{comink}{$\co$}
\psfrag{v}{$v$}
\psfrag{p}{$P$}
\psfrag{b}{$B^d$}
\psfrag{xu}{$\binom{x}{1}$}
\psfrag{x}{$x$}
\psfrag{hx}{$h$}
\psfrag{ps}{$P^*$}
\psfrag{hv}{$h_v$}
\psfrag{tps}{$\tilde{P}^*$}
\psfrag{c}{$\cM^2$}
\psfrag{cc}{$\co$}
\psfrag{H}{$\mathcal{H}^1$}

\includegraphics[width=0.8\hsize]{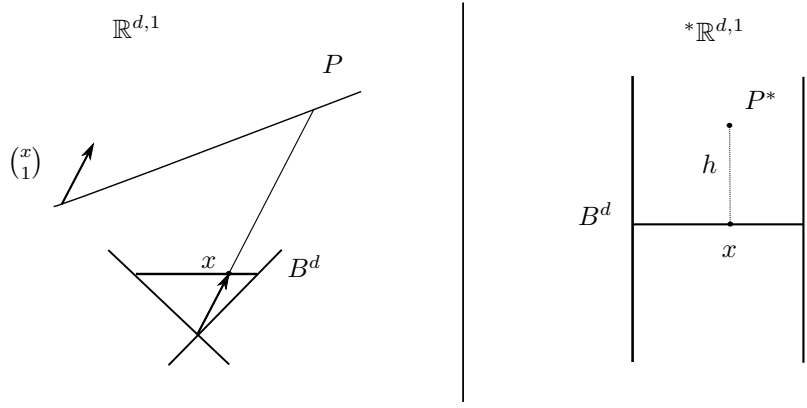}
\end{center}\caption{The dual $P^*$ of the hyperplane $P=\{y \in \R^{d,1} | \langle \binom{x}{1},y \rangle_{d,1} =h \} $.}\label{fig:cylindre}
\end{figure}

The interest of an affine model is essentially given by the following facts. The first one is an immediate consequence of the last point of Fact~\ref{fact geod com tilde}.

\begin{fact}
(Unparameterized) geodesics of $\co$ in the cylindrical model $B^d\times \R$ are (affine) geodesic segments.
\end{fact}

The second fact follows from Fact~\ref{fact hyp k tilde} and by construction.

\begin{fact}
The intersection of $B^d\times \R$ with any affine $k$-plane not containing a vertical line, with the metric induced by $g_{\co}$, is isometric to the hyperbolic space of dimension $k$.

In particular, $B^d\times \{0\}\cong B^d$  is the Klein ball model of the $d$-dimensional hyperbolic space.
\end{fact}

When $k=d$, we will call the intersection of $B^d\times \R$ with a $d$-plane not containing a vertical line a \emph{hyperbolic hyperplane}.

\begin{remark}{\rm
As every non-degenerate tangent plane of co-Minkowski space is isometric to the tangent plane of a hyperbolic space,  the sectional curvature of
co-Minkowski space is $-1$.
}\end{remark}

\subsubsection{Duality}\label{sec duality}

This cylindrical affine  model can be directly described from Minkowski space as follows. Let $P$ be an affine spacelike hyperplane of $\R^{d,1}$, and let
$(x,1)$ be a normal vector, with $x\in B^d$.
 Then there exists a number $h$ such that
$$P=\{y \in \R^{d,1} | \langle \binom{x}{1},y \rangle_{d,1} =h \} $$
and $P$ defines a point $P^*=(x,h)\in  B^d\times \R$, see Figure~\ref{fig:cylindre}.

Let us give more precisions about the ``duality'' between
Minkowski space and  co-Minkowski space. We already know that if $P$ is a spacelike hyperplane of Minkowski space, then $P^*$ is a point in $\co$. Conversely, if $P$ is a hyperbolic hyperplane of $\co$, let $P^*$ be the intersection of all the hyperplanes of Minkowski space whose duals are points in $P$.
For future reference, let us express this fact in terms of the cylindrical coordinates $B^d\times \R$.
\begin{fact}\label{equation plan dual}
Let $P$ be a hyperbolic hyperplane of co-Minkowski space, which is the graph of the affine function $h:B^d\to \R$, $h(x)=\langle x,v\rangle_d +c$. Then the point $P^*$ dual to $P$ has coordinates
$P^*=(v,-c)\in \R^d\times \R=\R^{d,1}$.

In other terms, if $P$ is a point of Minkowski space, then the hyperplane $P^*$ in co-Minkowski space is the graph of the affine map $h:B^d\to \R$, $h(x)=\langle P,\binom{x}{1}\rangle_{d,1}$.
\end{fact}
\begin{proof}
Let us fix $x\in B^d$. Then the point $X=(x,\langle v,x\rangle_d+c)\in B^d\times \R$ of co-Minkowski space belongs to $P$. Its dual is the spacelike hyperplane of Minkowski space defined as
$$X^*=\{(y,y_{d+1})\in \R^d\times \R | \langle \binom{x}{1},\binom{y}{y_{d+1}}\rangle_{d,1}=\langle v,x \rangle_d + c \} $$
i.e. $X^*=\{(y,y_{d+1})\in \R^d\times \R | \langle \binom{x}{1},\binom{y}{y_{d+1}}\rangle_{d,1}=\langle \binom{x}{1},\binom{v}{-c} \rangle_{d,1} \} $ and obviously $(v,-c)$ belongs to this hyperplane. As $x$ was arbitrary, $(v,-c)$ belongs to all the hyperplanes
dual to the points of $P$, that is the definition of $P^*$.

\end{proof}

The proof of the following facts are left to the reader.

\begin{fact}\label{dual plans}
\begin{enumerate}[nolistsep]
\item If $P$ is a hyperbolic hyperplane in co-Minkowski space $\co$, then $P^*$ is a point in Minkowski space $\R^{d,1}$ and $(P^*)^*=P$.
\item Let $P$ and $Q$ be two hyperbolic hyperplanes in $\co$.
\begin{enumerate}
\item if $P$ and $Q$ meet in $\co$ then $P^*$ and $Q^*$ are joined by a spacelike segment in $\R^{d,1}$.
\item if $P$ is strictly above $Q$  in  $\bar B^d\times \R$, then $Q^*-P^*$ is a future directed timelike segment in $\R^{d,1}$.
\item if $P$ and $Q$ have a common point in $\partial B^d\times \R$, then $P^*$ and $Q^*$ are joined by a lightlike segment.
\end{enumerate}
\end{enumerate}
\end{fact}

The vector space structure of Minkowski space corresponds via duality to the vector space structure on the space of restrictions to $B^d$ of affine maps.

\begin{fact}\label{fact vector dual}
Let $h_Q$ and $h_P$ be the restriction to $B^d$ of affine maps, such that their graphs are the hyperbolic hyperplanes $P, Q$ of co-Minkowski space, and let $\lambda\in \R$. Then the graph of $h_P+\lambda h_Q$ is dual to the point
$P^*+\lambda Q^*$ of Minkowski space.

\end{fact}

\begin{remark}\label{rem dual}{\rm
A convex spacelike hypersurface $S$ of Minkowski space is the boundary of the intersection of half-spaces bounded by spacelike hyperplanes. A hypersurface is \emph{F-convex} if it is the boundary of  a spacelike convex hypersurface such that any spacelike vector hyperplane is the direction of a support plane, and if the surface is in the future side of its support planes. Each support plane $P$ has a normal vector of the form $\binom{x}{1}$ for $x\in B^d$, so there is $h(x)\in \R$ such that $$P=\{y | \langle y,\binom{x}{1}\rangle_{d,1}=h(x) \}~.$$
The graph $S^*$ of the function $h$ in  $B^d\times \R$ is actually a convex hypersurface, see \cite{fv,BF}. In more classical terms, $h$ is the \emph{support function} of the convex set $K$ bounded by $S$:
\begin{equation}\label{eq support}h(x)=\operatorname{max}_{k\in K} \langle \binom{x}{1},k\rangle_{d,1}~. \end{equation}
Let us suppose furthermore that $S$ is the graph of a function $f:\R^d\to \R$. Then if $k\in K$ there is $y\in \R^d$ such that $k=\binom{y}{f(y)}$, and from \eqref{eq support},
$$h(x)=\operatorname{max}_{y\in \R^d} \langle\binom{x}{1},\binom{y}{f(y)}\rangle_{d,1}=
\operatorname{max}_{y\in \R^d}\{\langle x,y\rangle_d-f(y) \}~,$$
i.e. $h$ is nothing but the conjugate (Legendre--Fenchel dual) of $f$.

In the same way, convex hypersurfaces of Minkowski space which are in the past side of their support planes have dual hypersurfaces in the cylindrical model of co-Minkowski space, which are graphs of concave function $h:B^d\to \R$.
}\end{remark}

\begin{example}\label{ex hyp}{\rm

The dual surface of the hyperboloid $\{y | \langle y,y \rangle_{d,1}=-t^2, y_{d+1}>0 \} $
is the graph of the function $B^d\to  \R$, $x\mapsto -tL(x)$. Note that this function is convex (see \eqref{eq hess L}). In the same way,  dual surface of the hyperboloid $\{y | \langle y,y \rangle_{d,1}=-t^2, y_{d+1}<0 \} $ is the graph of the concave function $h(x)=tL(x)$.
}\end{example}

\begin{remark}\label{rem normal coord}{\rm Any hypersurface in Minkowski space which is an envelope of spacelike hyperplanes has a dual hypersurface in co-Minkowski space. This is more easily seen in the other way. For any $C^2$ function $h:B^d\to \R$, there exists a map $\chi:B^d \to \R^{d,1}$, the \emph{normal representation}, such that $P=\{y  | \langle y,\binom{x}{1}\rangle_{d,1}=h(x) \}$
is tangent to $\chi(B^d)$ at the point $\chi(x)$, see \cite[2.12]{fv}. Pay attention to the fact that $\chi$ is in general not a regular map, and that the concept of tangent hyperplane has to be understood in a generalized sense. The simplest example is when $h$ is the restriction to $B^d$ of an affine map: its graph is a hyperplane $P$ in the cylindrical model $B^d\times \R$ of co-Minkowski space, and  $\chi(B^d)$ is reduced to a point, the dual point of $P$ in Minkowski space.
}\end{remark}

\begin{remark}{\rm
The duality between $\R^{d,1}$ and $\co$ can also be seen
in $\R^{d+2}$, looking at $\R^{d,1}$ as a degenerate quadric in $\R^{d+2}$. See \cite[Section~2.5]{FS-hyp} for more details.
}\end{remark}

\subsubsection{Isometries in cylindrical coordinates}

Let us write the action of the isometry group of co-Minkowski space in the cylindrical coordinates $ B^d\times \R$. First let us state some facts about the action of hyperbolic isometries on $B^d$.  The group $\mbox{O}_+(d,1)$ acts by isometries on the hyperbolic
 space $\mathcal{H}^d$, and hence on the Klein ball model. More precisely, let $x\in B^d$ and $A \in \mbox{O}_+(d,1)$.
We will denote by $A\cdot x$ the image of $x$ by the isometry of the Klein ball model defined by $A$. We have
\begin{equation}\label{eq:klein}
\frac{1}{\left(A\binom{x}{1}\right)_{d+1}}A\binom{x}{1}=\binom{A\cdot x}{1}~.
\end{equation}

Note that as $A$ is a linear isometry of Minkowski space $\R^{d,1}$, we have
$$|\left(A{\textstyle \binom{x}{1}}\right)_{d+1} |^2(\| A\cdot x \|^2 - 1)=\|x\|^2-1 $$
i.e.
\begin{equation}\label{eq: LL affine}\left(A{\textstyle\binom{x}{1}}\right)_{d+1}=\frac{L(x)}{L(A\cdot x)}~,\end{equation}
so, together with \eqref{eq:klein}, one obtains
\begin{equation}\label{A klein}
A\binom{x}{1}=\frac{L(x)}{L(A\cdot x)}\binom{A\cdot x}{1}~.
\end{equation}

For simplicity, let us  fix also the following coordinate system; every element $(x_1, \ldots , x_{d+1})$ of $\R^{d+1}$ has
a horizontal component $\bar{x}=(x_1, \ldots , x_d)$ and a vertical component $x_{d+1}$.
If  $\langle \bar{x}, \bar{y}\rangle_d$ is the scalar product of horizontal elements, we have, for $x,y\in \R^{d,1}$,
$\langle x,y\rangle_{d,1}=\langle \bar{x},\bar{y}_{d}\rangle_d-x_{d+1}y_{d+1}$.

\begin{lemma}
Let $(x,h)\in B^d\times \R$ and $(A,v)\in  O_+(d,1)\ltimes \R^{d,1}$.
Then the isometry of co-Minkowski space defined by $(A,v)$ acts on the cylindrical coordinates as follows:
\begin{equation}\label{action isom}(A,v)(x,h)=\left(A\cdot x, \frac{L(A\cdot x)}{L(x)}h+\langle A\cdot x,\bar{v}\rangle_{d}-v_{d+1} \right)~.\end{equation}

\end{lemma}
\begin{proof}

When the isometry is linear, i.e. when $v=0$, the elements of the image of $(x,h)$ by $(A,v)$ are
elements of $\R^{d,1}$ satisfying:
\begin{eqnarray*}
  h &=& \langle \binom{x}{1}, A^{-1}\binom{y}{y_{d+1}} \rangle_{d,1} \\
   &=& \langle A\binom{x}{1}, \binom{y}{y_{d+1}} \rangle_{d,1} \\
   & \stackrel{\eqref{A klein}}{=} & \langle \frac{L(x)}{L(A\cdot x)}\binom{A.x}{1}, \binom{y}{y_{d+1}} \rangle_{d,1}~.
\end{eqnarray*}

Therefore, the image of $(x, h)$ by $(A, 0)$ is $(A\cdot x,\frac{L(A\cdot x)}{L(x)}h)$.

In the case of a translation by a vector $v = \binom{\bar{v}}{v_d}$ we have:
\begin{eqnarray*}
  h &=& \langle \binom{x}{1}, \binom{y}{y_{d+1}} - \binom{\bar{v}}{v_{d+1}}\rangle_{d,1} \\
   &=& \langle \binom{x}{1}, \binom{y}{y_{d+1}} \rangle_{d,1} - \langle x, \bar{v} \rangle_d + v_{d+1}~.
\end{eqnarray*}

Hence the image of $(x,h)$ by the translation is
$$(x,h + \langle x, \bar{v} \rangle_d - v_{d+1})~.$$
The Lemma follows because from \eqref{group structure}, $(A,v)=(\operatorname{Id},v)(A,0)$.
\end{proof}

\begin{remark}\label{rem inv tL}{\rm
There is an easy way to see the action of $O_+(d,1)$ in the coordinates $B^d\times \R$. Actually,
$ B^d\times \R$ is foliated by the graphs of the functions $tL$, $t\in \R$.
Note that those graphs are, for $t\not= 0$, the duals of the two-sheeted hyperboloids centered at the origin in Minkowski space, see Example~\ref{ex hyp}. Observe that for the sheet with positive (respectively negative)  $x_{d+1}$, the parameter $t$ is negative (respectively positive).
Hence if $(x,h)\in B^d\times \R$ belongs to the graph of $tL$ for some $t$,
then for any $A\in O_+(d,1)$, $(A,0)(x,h)$ still belongs to the graph of $tL$, and of course its projection onto $B^d\times\{0\}$ is
$(A\cdot x, 0)$, see Figure~\ref{fig action}.
}\end{remark}

\begin{figure}
\begin{center}
\psfrag{b}{$B^d$}
\psfrag{x}{$x$}
\psfrag{ax}{$A\cdot x$}
\psfrag{hx}{$(x,h)$}
\psfrag{ahx}{$(A,0)(x,h)$}
\psfrag{h2x}{$(x,h')$}
\psfrag{ah2x}{$(A,0)(x,h')$}

\includegraphics[width=0.3\hsize]{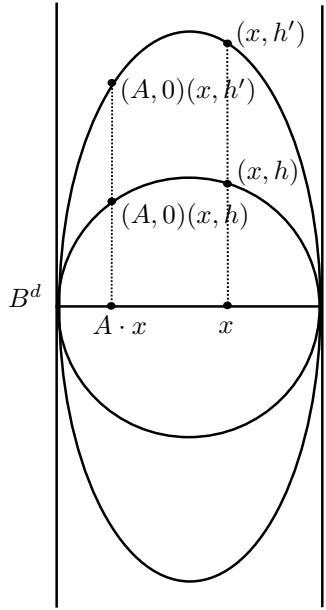}
\end{center}\caption{Action of $(A,0)$ on $B^d\times \R$.}\label{fig action}
\end{figure}

\begin{remark}{\rm
In order to fully understand the action of $O(d,1)$ onto co-Minkowski space, we have to describe the action of $-\operatorname{Id}\in O(d,1)$ onto
$B^d\times \R$.  It is actually straightforward that
\begin{equation}\label{-id}(-\operatorname{Id},0)(x,h)=(x,-h)~. \end{equation}
}\end{remark}

We now describe the action of the isometries of co-Minkowski space on functions.
Let $S$ be a hypersurface in Minkowski space which is the graph of a map $h:B^d\to\R$. Then, for $(A,v)\in O_+(d,1)\ltimes \R^{d,1}$, due to \eqref{action isom}, the hypersurface $(A,v)S$ is the graph of the map $(A,v)h:B^d\to\R$ defined as
\begin{equation}\label{def action h}(A,v)h(x):= \frac{L( x)}{L(A^{-1}\cdot x)}h(A^{-1}\cdot x)+\langle x,\bar{v}\rangle_{d}-v_{d+1}~. \end{equation}

\begin{lemma}\label{lem:action hessien}
Let $h:B^d\to \R$ be a $C^2$ map and   $(A,v)\in O_+(d,1)\ltimes \R^{d,1}$.
Then
$$\operatorname{Hess}[(A,v)h](x)(X,Y)=\frac{L(x)}{L( A^{-1}\cdot x)}\operatorname{Hess} h (A^{-1}\cdot x)(DA^{-1}(x)X,DA^{-1}(x)Y)~. $$
\end{lemma}
\begin{proof}
As $(\operatorname{Id},v)h$ is the sum of $h$ with an affine function, we clearly have $\operatorname{Hess}[(\operatorname{Id},v)h](x)=\operatorname{Hess} h (x)$. So we need to check the result only for $(A,0)$.
As
$$\operatorname{Hess}[(A,0)h] = \operatorname{Hess} \left( \frac{L}{L\circ A^{-1}}(h\circ A^{-1})\right)~,$$
the result follows from the rules \eqref{eq: hessien produit}
and
\begin{equation}\label{hessien compose}\operatorname{Hess}(f\circ g)(x)(X,Y)=\operatorname{Hess}f(g(x))( \operatorname{d}g(x)(X), \operatorname{d}g(x)(Y))+ \operatorname{d}f(g(x))(\operatorname{Hess}g(x)(X,Y))~, \end{equation}
 using the two following facts during the computations:
\begin{itemize}[nolistsep]
\item  $\frac{L}{L\circ A}$ is an affine map by \eqref{eq: LL affine}, so has null Hessian;
\item Differentiating two times \eqref{eq:klein} we obtain
$$A{\textstyle\binom{X}{0}}_{d+1}DA(x)(Y) + A\textstyle{\binom{Y}{0}}_{d+1}DA(x)(X)+A\textstyle{\binom{x}{1}}_{d+1}\operatorname{Hess}A(x)(X,Y)=0~,   $$
so using \eqref{eq: LL affine} again,
$$ \operatorname{d}\frac{L}{L\circ A}\otimes  \operatorname{d} A +  \operatorname{d}A\otimes  \operatorname{d}\frac{L}{L\circ A}+ \frac{L}{L\circ A}\operatorname{Hess}A=0~.   $$
\end{itemize}
\end{proof}

\begin{lemma}\label{iso convexe}
Let $h:B^d\to \R$ be a convex map. Then for $(A,v)\in O_+(d,1)\ltimes \R^{d,1}$, $(A,v)h$ is a convex map.
\end{lemma}
Note that from \eqref{-id}, $(-\operatorname{Id},0)h$ is concave if $h$ is convex.
\begin{proof}
The simplest way to see this is to argue that the dual of the  epigraph of $h$ is a future convex set in Minkowski space, see Remark~\ref{rem dual}. The isometry $(A,v)$ will send
this future convex set to a future convex set (because $A\in O_+(d,1)$), whose support function is exactly  $(A,v)h$, hence convex.
\end{proof}

\subsubsection{Connection in cylindrical coordinates}\label{con}

Clearly,
the restriction of the vector field  $\frac{\partial }{\partial x_t}=(0,\ldots,0,1)$ of $\R^{d+2}$ to
$\cM^{d+1}$ is invariant under the action of the isometries of $\cM^{d+1}$. It is also immediate to see that $\frac{\partial }{\partial x_t}$ is parallel:  $\nabla^{\mathscr{co}\mathcal{M}^{d+1}} \frac{\partial }{\partial x_t}=0$. We will denote by $\T$ the image of
$\frac{\partial }{\partial x_t}$ in co-Minkowski space. An elementary computation (see Figure~\ref{fig Tb}) shows that in the cylindrical coordinates
$B^d\times \R$,
\begin{equation}\label{def T}\T=L\frac{\partial}{\partial x_{t}}~. \end{equation}

\begin{figure}
\begin{center}
\psfrag{O}{$0$}
\psfrag{1}{$\displaystyle\frac{\partial}{\partial x_{t}}$}
\psfrag{T}{$\T(x,h)$}
\psfrag{x}{$(x,h,1)$}
\psfrag{o}{$\T$}
\psfrag{Lx}{$L^{-1}(x)(x,h,1)$}

\includegraphics[width=0.5\hsize]{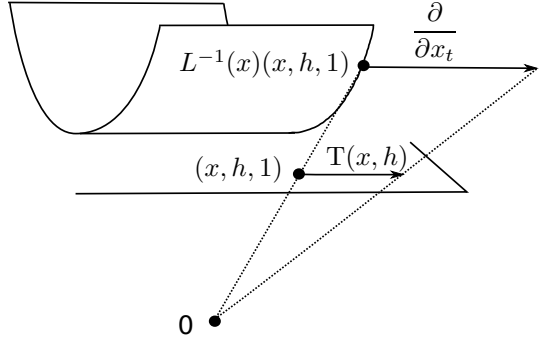}
\end{center}\caption{By Thales theorem, $\T=L\frac{\partial}{\partial x_t}$.}\label{fig Tb}
\end{figure}

In particular, $\T$ is invariant under the action of $O_+(d,1)\ltimes \R^{d,1}$, and  $\T$ is parallel: $\nabla^{\co}\T =0$. Observe that the trajectories of the flow generated by $\T$ are the vertical lines. In Minkowski space, the flow generated by $\T$ corresponds to parallel displacement of spacelike hyperplanes.

With the help of $\T$, one can express the connection $\nabla^{\co}$ in the cylindrical
 coordinates. Namely, at each point $(x,h)\in  B^d\times \R$, we set $\T(x,h)$ as the vector basis for the $\R$-component of the tangent space.
Hence a vector field $X$ of $ B^d\times \R$ can be written
$X=X_h+X_\T\T$, with $X_h\in T_xB^d$ and
$X_\T\in \R$. If  $Y$ is another vector field of $B^d\times \R$, then
\begin{equation}\label{def con}\nabla^{\co}_YX= \nabla_{Y_h}X_h + Y_h(X_\T)\T +Y_\T[\T,X]~. \end{equation}
This is easily checked using the definition of the connection  $\nabla^{\co}$ and the fact that $\T$ is parallel.

\subsubsection{Volume form}\label{sec volume form}

For future reference, let us mention that a volume form $\omega_{\mathscr{co}\mathcal{M}^{d+1}}$ is also given on
 $\cM^{d+1}$. For $v_1,\ldots, v_{d+1}$  vectors of $\R^{d+2}$ tangent to
 $\cM^{d+1}$, set $$ \omega_{\mathscr{co}\mathcal{M}^{d+1}}  (v_1,\ldots,v_{d+1}):=\omega_{\R^{d+2}}(v_1,\ldots,v_{d+1},\operatorname{N})~$$
 (recall that $\operatorname{N}$ is the vector field $\operatorname{N}(x)=x$ on $\cM^{d+1}$).
This  form is invariant under orientation preserving isometries and parallel for
 $\nabla^{\mathscr{co}\mathcal{M}^{d+1}}$. It induces a parallel  form
 $\omega_{\co}$ on co-Minkowski space, invariant under orientation preserving isometries, and called the \emph{volume form} of co-Minkowski space.

In the cylindrical coordinates,  $\omega_{\co}$  is defined as follows.
At a point of
$ B^d\times \R$, let $v_1,\ldots,v_d$ be an oriented free family of non-vertical tangent vectors. In particular, $v_1,\ldots,v_d$ are tangent to a hyperbolic hyperplane, so, keeping the same notation, we can consider a family $v_1,\ldots,v_d$ of oriented orthonormal vectors fields, such that
$v_1,\ldots,v_d,\T$ is positively oriented. Then $\omega_{\co}$ is the unique $(d+1)$-form which is equal to $1$ when evaluated at such a family of vectors.

\subsection{Extrinsic geometry of graphs}\label{sec:extrin}

Let $h:B^d\to \R$ be a $C^2$ map. Its graph $S$ is a hypersurface in
$B^d\times \R$, hence in co-Minkowski space if one uses the cylindrical coordinates.
Note that  the graph is always transverse to the vertical vector field $\T$ defined by \eqref{def T}, so the metric induced on
$S$ by the ambient degenerate metric $g_{\co}$ of co-Minkowski space is  always a hyperbolic metric, that does not give too much informations. But still, some informations can be obtained from the extrinsic geometry of $S$.
To do so, we will consider the vector field $\T$  as the normal vector to $S$.

\subsubsection{Second fundamental form and mean curvature}\label{sec: extr}

Let $h:B^d\to \R$ be a $C^2$ map and let $S$ be its graph.
Any vector field of $S$ can be written
$X+\operatorname{d}h(X)L^{-1}\T$, where $X$ is a
 vector field of $B^d$.

\begin{fact} For any smooth vector field $X$ on $B^d$ and
$C^2$ map $h:B^d\to \R$,
\begin{equation}\label{con surf}\nabla^{\co}_{(Y +L^{-1}\operatorname{d}h(Y)\T)}(X+L^{-1}\operatorname{d}h(X)\T)=\nabla^{\H^d}_YX + L^{-1}\operatorname{d}h(\nabla^{\H^d}_YX) \T + L^{-1}\hess h(X,Y) \T~. \end{equation}
\end{fact}
\begin{proof}
First let $k\in \{1,\ldots,d\}$. As $X$ does not depend on the
$\frac{\partial}{\partial x_{t}}$ direction, and as $\T^k=0$,
$[\T,X]^k=\T^i\frac{\partial X^k}{\partial x_i} - X^i\frac{\partial \T^k}{\partial x_i}=0, $
and
$[\T,X]^{t}= - X^i\frac{\partial \T^{t}}{\partial x_i} = -X(L). $
Also, as $L^{-1}\operatorname{d}h(X)$ does not depend on the vertical coordinate,  $[\T, L^{-1}\operatorname{d}h(X)\T]=0$.
At the end of the day, if we are at a point $x\in B^d$,

$$[\T,X+L^{-1}\operatorname{d}h(X)\T ]=-X(L)\frac{\partial }{\partial x_{t}}= -X(L)L^{-1}\T=L^{-2}\langle x,X\rangle\T~. $$

 So from \eqref{def con},
$$\nabla^{\co}_{(Y+L^{-1}\operatorname{d}h(Y)\T)}(X+L^{-1}\operatorname{d}h(X)\T)=\nabla_YX + (Y(L^{-1}\operatorname{d}h(X))+L^{-3}\operatorname{d}h(Y)\langle x,X\rangle_d)\T ~. $$

We have $Y(L^{-1}\operatorname{d}h(X))=L^{-1}(Y(X(h))+Y(L^{-1})\operatorname{d}h(X)$, and from \eqref{hess con}, $L^{-1}(Y(X(h))=L^{-1}\operatorname{Hess}^{\H^d}h(X,Y)+L^{-1}\operatorname{d}h(\nabla_YX)$. Also, if we are at the point $x$, $Y(L^{-1})=\langle x,Y\rangle_d L^{-3}$:

$$\nabla^{\co}_{(Y+L^{-1}\operatorname{d}h(Y)\T)}(X+L^{-1}\operatorname{d}h(X)\T)=\nabla_YX +L^{-1}\operatorname{d}h(\nabla_YX)\T + L^{-1}\mathcal{X}\T $$
with $\mathcal{X}=\operatorname{Hess}^{\H^d}h(X,Y)+L^{-2}\langle x,Y\rangle_d \operatorname{d}h(X)+L^{-2}\langle x,X\rangle_d \operatorname{d} h(Y)$, and by \eqref{nabla2}, $\mathcal{X}=\operatorname{Hess} h(X,Y)$.
\end{proof}

 Given two vector fields tangent to $S$, the graph of $h$, then their co-Minkowski connection decomposes as a part tangent to $S$, and a part colinear to $\T$, where $\T$ may be think as a unit normal vector field to $S$.
Mimicking  the classical theory of surfaces, we define the \emph{second fundamental form} $\II_h$ of $S$ as the colinearity factor.
More precisely, Equation \eqref{con surf} says that
for $x\in B^d$ and $X,Y\in T_xB^d$,
\begin{equation}\label{def II}\II_h(x)(X,Y)=L^{-1}(x)\hess h (x)(X,Y)~.\end{equation}

\begin{remark}\label{rem II cod}{\rm
From Lemma~\ref{lem tens cod}, the second fundamental form is a symmetric Codazzi tensor on $\H^d$, and any symmetric Codazzi tensor on the hyperbolic space is the second fundamental form of a unique hypersurface in co-Minkowski space. This is a kind of ``fundamental theorem for hypersurfaces'' in co-Minkowski space, with the condition about the first fundamental form reduced to the hypothesis that the metric is hyperbolic. Note that here there is no Gauss condition, i.e. for $d=2$ there is no  relation between the curvature of the induced metric and the determinant of the second fundamental form.
}\end{remark}

The \emph{shape operator} $\shape(h)$ of $S$ is
the symmetric linear mapping associated to the second fundamental form by the hyperbolic metric: $\II_h(X,Y)=g_{\H^d}(\shape(h)(X),Y)$.
From \eqref{eq hyp hess}, if $\grad^{\H^d}$ is the gradient for $g_{\H^d}$, we have
$$\shape(h)(X)=\nabla^{\H^d}_X \grad^{\H^d}(L^{-1}h) - (L^{-1}h)X~. $$


The \emph{mean curvature} $\mean (h)$ of the graph of $h$ is the trace for the hyperbolic metric of the shape operator  times $1/d$. From the definition or  Fact~\ref{facts laplacian}, it can be written in different ways:
with the help of the the Euclidean
Laplacian $\Delta$
\begin{equation}\label{mean eucl}
\mean(h)(x)=\frac{1}{d} \operatorname{Tr}_{\gh} \left(L^{-1}\hess h\right)(x)=\frac{1}{d}L(x)(
\Delta h(x) - \hess h(x)(x,x))~,
\end{equation}
or with the help of the hyperbolic Laplacian $\Delta^{\H^d}$
\begin{equation}\label{mean hyp}\mean(h)(x)=\frac{1}{d}\Delta^{\H^d} (L^{-1}h)(x) - (L^{-1}h)(x)~.
\end{equation}

\begin{remark}\label{rem H positif}
{\rm  Let us suppose that $h:B^d\to \R$ is $C^2$ and convex. Using a basis of eigenvectors, it follows from
\eqref{mean eucl} that $\mean(h)$ is non-negative, and that if $\mean(h)=0$ then $h$ is affine.
}\end{remark}

\begin{proposition}[{\cite{li95}}]\label{lem Li}
If the graph of a $C^2$ convex function $h:B^d\to \R$ has its mean curvature bounded from above, then $h$  has a  continuous extension to $\bar B^d$.
\end{proposition}
\begin{proof}
Suppose that there is $C$ such that for any $x\in B^d$,
$\mean(h)(x) < C$. Let $\theta \in \partial B^d$, and let $h_\theta$ be the restriction of $h$ to the segment parametrized by $r\in [0,1[$ from the origin to $\theta$. Let us also denote
$l(r)=\sqrt{1-r^2}$. By \eqref{mean eucl},
$$h_\theta''(r) < C l(r)^{-3}~. $$

For $1/2<r<1$, we write
$$h_\theta'(r) \leq h_\theta'(1/2) +  C \int_{1/2}^{r} l^{-3}$$
and as, for $1/2<t<1$,  $(1-t^2)^{-1}< (1-t)^{-1}$, we have
\begin{equation}\label{eq:int maj}h_\theta'(r) < h_\theta'(1/2)+ 2 C (1-r)^{-1/2}~.\end{equation}

Also, as $h$ is convex, $h_\theta$ is convex, hence for $1/2<r<1$,
\begin{equation}\label{derivee croissante}
h'_\theta(1/2) \leq h'_\theta(r)~.
\end{equation}

Let us define
$$g(\theta)=\int_{1/2}^1 h_\theta'  - h_\theta(1/2)~. $$

As $\int_{1/2}^1 (1-r)^{-1/2}\mbox{d}r$ is finite, by \eqref{eq:int maj} and \eqref{derivee croissante}, $g(\theta)$ is well defined.
Also, together with \eqref{eq:int maj}, \eqref{derivee croissante} and the Dominated convergence theorem, $g$ is continuous.
\end{proof}

\subsubsection{Mean surfaces}\label{sec minmax}

\begin{definition}
A hypersurface $S$ of co-Minkowski space is called \emph{\minmax} if  it is the graph of a
$C^2$ function $h:B^d\to \R$ with $\mean(h)=0$.

Abusing terminology, the function $h$ itself may be also called \minmax.
\end{definition}

 Note that when $d=2$, the \minmax surface is not critical for the area functional, as all the graphs of functions $B^2\to \R$ in co-Minkowski space have the same area form (because they are all isometric to the hyperbolic plane).

Due to \eqref{mean eucl}, $h$ is \minmax if and only
if for any $x\in B^d$, $\Delta h(x) - \hess h(x)(x,x)=0$. This is an elliptic equation with only second-order terms, that allows to apply strong results of PDE theory. For this, we have to consider boundaries conditions.

\begin{definition}
Let $b:\partial B^d\to \R$ be a continuous map.  A continuous function $h:B^d\to\R$ is called a \emph{$b$-map} if it extends continuously as $b$ on $\partial B^d$.
\end{definition}

\begin{proposition}\label{prop critical gal}
For any continuous function $b:\partial B^d \to \R$, there is a unique $C^\infty$ smooth
\minmax $b$-map, denoted by $h^{\crit}_b$.
\end{proposition}
\begin{proof}
The uniqueness is classical from the ellipticity of $L^{-1}\mean$  \cite[Theorem~3.3]{GT}.
Existence follows from the
fact that the elliptic equation $\Delta f(x) - \hess f(x)(x,x)=0$
has only second-order terms and that the domain is a ball, see \cite[Corollary~6.24']{GT}. Dividing the equation by $L^2$, we obtain a strictly elliptic equation, and regularity theorems apply, e.g. \cite[Corollary~8.11]{GT}.
\end{proof}

\begin{lemma}\label{conv min}
If $b_n:\partial B^d \to  \R$ are continuous functions uniformly converging to  $b:\partial B^ d\to \R$, then $h_{b_n}^{\crit}$  is converging to $h_b^{\crit}$.
\end{lemma}
\begin{proof}
Let $b_n$ such that the supremum of $|b_n-b|$ is arbitrarily small. Then $\mean(h^{\crit}_{b_n}-h^ {\crit}_b)=0$,  with boundary data $b_n-b$. By  the maximum principle \cite[Theorem~3.1]{GT}, $h^{\crit}_{b_n}-h^ {\crit}_b$ is arbitrarily small. The same conclusion holds for $h^ {\crit}_b-h^{\crit}_{b_n}$.
\end{proof}

\begin{remark}{\rm
For a continuous map $b:\partial B\to \R$, it is possible to associate to $h^\crit_b$ a (non-regular and non convex) dual hypersurface in Minkowski space,  see
Remark~\ref{rem normal coord}. For $d=2$, at points of regularity, this surface has zero mean curvature. We refer to  \cite{fv} for more details.
}\end{remark}

\subsubsection{Convex hull}\label{sec convex hull}

Let $b:\partial B^d\to \R$ be a continuous map.  Let
$$\mathcal{A}_b=\{a | a:\R^d\to\R \mbox{ is an affine function and }a|_{\partial B^d}\leq b \} $$
and for $x\in B^d$, let us define
\begin{equation}\label{eq:h-}h_b^-(x) := \sup \{ a(x) |
a \in \mathcal{A}_b \}~,\end{equation}
and
\begin{equation}\label{def h-}h_b^+(x):=-h_{-b}^-(x)~.\end{equation}

\begin{proposition}\label{prop h+}
For any $x\in B^d$, $h_b^-(x)$  defines
a convex $b$-map $h_b^-:B^d\to \R$.
Moreover, if $h:B^d\to \R$ is a convex $b$-map, then $h_b^-\geq h$.

For any $x\in B^d$, $h_b^+(x)$  defines
a  concave $b$-map $h_b^+:B^d\to \R$.
Moreover, if $h:B^d\to \R$ is a concave $b$-map, then $h_b^+\leq h$.

In general, we have $h_b^+ \geq h_b^-$.
If $h_b^+=h_b^-$, then $b$ is the restriction to $\partial B^d$ of an affine map of $\R^d$.
\end{proposition}
\begin{proof}
The properties of $h^-_b$ are proved in the proof of Theorem~1.5.2 in \cite{Gu01}. The properties of $h^+_b$ then follows immediately from \eqref{def h-}. The last property is then obvious, as affine maps are the only ones being in the same time convex and concave.
\end{proof}

Let $\Lambda(b)$ be the graph of $b:\partial B^d\to \R$
in $\partial B^d\times \R$, and let
 $\CH(b)$ be
 the  affine convex hull  of $\Lambda(b)$ in $\R^{d+1}$, that is, the smallest  convex set of $\R^{d+1}$ containing $\Lambda(b)$.
Note that as $\bar B^d\times \R$ is a convex set containing
$\Lambda(b)$, then $\CH(b)\subset \bar B^d\times \R$.

\begin{lemma}
The boundary of  $\CH(b)$ is the union of the graphs of
$h_b^+$ and $h_b^-$.
\end{lemma}
\begin{proof}
This follows from the definitions of $h_b^+$ and $h_b^-$, because
 $\CH(b)$ is the intersection of all the half-spaces containing $\Lambda(b)$.
\end{proof}

The set $\CH(b)$ satisfies the local geodesic property: for any $x\in \CH(b)\setminus \Lambda(b)$, $x$ lies in an open segment contained in $\CH(b)\setminus \Lambda(b)$ \cite[Theorem~4.19]{smith}.

\begin{lemma}\label{lem:crit conv}
The \minmax surface given by the boundary condition $b:\partial B^d\to \R$  is contained in the convex hull $\CH(b)$:
$$h_b^-\leq h^{\crit}_b \leq h_b^+~.$$
\end{lemma}
\begin{proof}
Let $a\in \mathcal{A}_b$. By the maximum principle \cite[Theorem~3.1]{GT},
$a-h^{\crit}_b$ attains its maximal value on $\partial B$. But on $\partial B^d$,
$a\geq b$, so on $B^d$, $a-h^{\crit}_b\leq  a|_{\partial B^d}-b = b-b =0$, i.e. $a\leq h^{\crit}_b$. Then by definition of $h^-_b$,
$h^-_b \leq h_b^{\crit}$. Similarly, one proves that
$ h_{-b}^- \leq h_{-b}^{\crit}=-h_b^{\crit}$ i.e.
$h_b^+=-h_{-b}^- \geq h_b^{\crit}$.
\end{proof}

\begin{lemma}\label{le:btau}
If $(b_n)_{n \in \N}$ is a sequence of continuous functions from $\partial B^d$ into $\R$ converging uniformly to  $b:\partial B^d\to \R$, then $(h^-_{b_n})_{n \in \N}$ (resp. $(h^+_{b_n})_{n \in \N}$)  is converging to $h^-_b$ (resp. $h^+_b$).
\end{lemma}
\begin{proof}
Let $\epsilon >0$ and $x\in B^d$. Then there exists an affine function $a$ such that $h_b^-(x) \geq a(x)$, $a(x)+\epsilon \geq h_b^-(x)$ and $a|_{\partial B^d} \leq b$.
In particular, for $n$ large enough,
$a|_{\partial B^d}-\epsilon \leq b_n$.
As $a|_{\partial B^d}-\epsilon $ is an affine function,
then $h^-_{b_n}(x)\geq a(x)-\epsilon$. As $a$ was chosen such that  $a(x)+\epsilon \geq h_b^-(x)$, then
$h^-_{b_n}(x)+2\epsilon \geq h_b^-(x)$. A similar conclusion holds, exchanging the roles of $b$ and $b_n$.
\end{proof}

\begin{remark}\label{rem domain of dependance}{\rm
The dual in Minkowski space of  the epigraph of a convex $b$-map is a convex set. Its \emph{domain of dependence}, or \emph{Cauchy domain}, denoted by $\Omega_b^-$, is the interior of the intersection of the future side of all the lightlike hyperplanes containing it.
This intersection is nothing but the dual of the
the epigraph of $h^-_b$.
The domain of dependence $\Omega_b^-$ is future complete. Considering $h^+_b$ instead of $h^-_b$, and concave figures instead of convex ones, we obtain the domain of dependence $\Omega_b^+$. See Figure~\ref{fig:mean distance} and  \cite{Bar05,Bon05} for more details.
}\end{remark}

\begin{remark}{\rm
The function $h^\crit_b$ is the solution of the Dirichlet problem for an elliptic  linear equation. The convex function
$h_b^-$ is the solution of the Dirichlet problem for the
Monge--Amp\`ere equation, see \cite{Gu01}.
}\end{remark}

\subsubsection{The mean curvature measure}\label{sec mean curv measure}

For a $C^2$ function $h:B^d\to \R$, we have defined in Section~\ref{sec: extr}
the mean curvature function, which
is non-negative if $h$ is convex by Remark~\ref{rem H positif}.
For a convex $C^2$ function $h:B^d\to \R$, let us define the
\emph{mean curvature measure}
$$\meanm(h)= d\mean(h)\omega_{\H^d}~, $$
where $\omega_{\H^d}$ is the volume form given by the hyperbolic metric on $B^d$. By
\eqref{eq:volume} and \eqref{mean eucl},
for any $\varphi\in C^0_0(B^d)$ (here the subscript $0$ means ``with compact support''),
$$\meanm(h)(\varphi)=\int_{B^d}\left(\Delta h(x) - \operatorname{Hess} h(x)(x,x)\right)L^{-d}(x)\varphi(x) \operatorname{d}x~. $$

If moreover $\varphi \in C^\infty_0(B^d)$, by integration by part:

\begin{equation}\label{eq measure mean}\meanm(h)(\varphi)=\int_{B^d}\left(\Delta \varphi(x)  - \operatorname{Hess} \varphi(x)(x,x)\right) h(x) L^{-d}(x) \operatorname{d}x~. \end{equation}

For any convex function $h:B^d\to \R$, let us define
$\meanm(h)$ as the linear form on $C_0^\infty(B^d)$ defined by \eqref{eq measure mean}. On any compact ball $K$ contained in $B^d$, by standard
convolution, one can find a sequence $(h_i)_{i \in \N}$ of $C^\infty$ convex functions uniformly approximating $h$.
For any $C^\infty$ function $\varphi$ whose support is included in $K$, we clearly have
$\meanm(h_j)(\varphi)\to \meanm(h)(\varphi)$.
As $\meanm(h_j)$ is a measure, it is also a distribution, and
the preceding limit says that $\meanm(h)$ is also a distribution on $K$ \cite[Theorem~2.1.8]{hormander}.
Actually, as  the $\meanm(h_j)$ are measures, then
$\meanm(h)$ is a measure on $K$ \cite[Theorem~2.1.9, Theorem~2.1.7]{hormander}. Changing $K$ and using the localization property of distribution \cite[Theorem~2.2.4]{hormander}, it follows that $\meanm(h)$ is a measure on $B^d$.
More precisely, $\meanm(h)$ is a Radon measure on $B^d$.

The following result is given by \cite[Theorem~2.1.9, Theorem~2.1.7]{hormander}.

\begin{lemma}\label{lem conv measures}
Let $(h_n)_{n \in \N}$ be a sequence of convex functions from $ B^d$ into $\R$
converging to a convex function $h: B^d \to \R$.
Then the sequence of measures $(\meanm(h_n))_{n \in \N}$ weakly converges to $\meanm(h)$.
\end{lemma}

Recall the action of isometries on functions defined by
\eqref{def action h}. Recall also from Lemma~\ref{iso convexe} that if $h$ is convex, then $(A,v)h$ is convex for $(A,v)\in O_+(d,1)\ltimes \R^{d,1}$.

\begin{lemma}\label{lem: action sur mesure}
Let $\varphi\in C^0_0(B^d)$ and $(A,v)\in O_+(d,1) \ltimes \R^{d,1}$. Then:
$$\meanm((A,v)h)(\varphi)=\meanm(h)(\varphi \circ A)~.$$
\end{lemma}
\begin{proof}

We will prove the result for a $C^2$ function $h$, the general result follows by approximation.
In the $C^2$ case, the result follows because by definition
$$\meanm(h)(\varphi \circ A)=\int_{B^d} (\varphi \circ A)(x) (\operatorname{Tr}_{\gh} L^{-1}\hess h)(x) \mbox{d}\omega_{\H^d}(x) $$
so by a change of variable, as $A$ is a hyperbolic isometry,
$$\meanm(h)(\varphi \circ A)=\int_{B^d} \varphi(x) (\operatorname{Tr}_{\gh} L^{-1}\hess h)(A^{-1}\cdot x) \mbox{d}\omega_{\H^d}(x)~, $$
and by
Lemma~\ref{lem:action hessien},
$$(\operatorname{Tr}_{\gh} L^{-1}\hess h)(A^{-1}\cdot x)=
(\operatorname{Tr}_{\gh} L^{-1}\hess [(A,v)h])(x)~. $$

\end{proof}

\subsubsection{The fundamental  example  of a wedge}

Let us consider an elementary example to give a geometric insight on the mean curvature measure introduced in the previous section. This example will make clear that, for well-chosen convex functions, this measure is a kind of "pleating measure", similar
	to the notion developed by Thurston for isometric pleated embeddings of hyperbolic surfaces in the $3$-dimensional hyperbolic space, see sections \ref{simlicial} and \ref{cas 2+1}.

Let $l$ be the intersection of $B^d$ with  an affine hyperplane  of
$\R^{d}$, which separates $B^d$ into two connected components $l^-$ and $l^+$, where $l^-$ is the component containing the origin $0$ of the coordinates of $\R^{d}$. Let $p_l$ be the (Euclidean) orthogonal projection of $0$ onto $l$, and let $n_l=p_l/\|p_l\|$. If $l$ is a vector hyperplane, then $l^-$ is chosen arbitrarily, and $n_l$ is the (Euclidean) unit normal vector pointing to $l^+$.

\begin{definition}\label{def can map}
The \emph{canonical map} $h_l:B^d\to \R$ associated to $l$ is
defined as $h_l(x)=\frac{1}{L(p_l)}\langle x-p_l,n_l\rangle$.
\end{definition}

Observe that $h_l$ is an affine map vanishing on $l$.
Let $\mathbf{1}_A$ be the indicator function of a set $A$.

\begin{definition}
A \emph{wedge} on a hyperplane $l$ is a continuous map $h: B^d\to \R$
of the form $h=h_-+(h_+-h_-)\mathbf{1}_{l^+}$ where $h_-,h_+$ are two affine maps.

The \emph{angle} of a wedge (in the co-Minkowski sense) is the unique real number $\alpha$ such that, with the notations above,
\begin{equation}\label{def angle}h_+-h_-=\alpha h_l~.\end{equation}
\end{definition}

The wedge is therefore a piecewise affine map, admitting $l$ as a locus of non-differentiability (if the angle is nonzero).

\begin{fact}\label{fact convex wedge}
A wedge is convex and different from an affine map if and only its angle is positive.
\end{fact}
\begin{proof}
By definition, $h_l$ is positive on $l^+\setminus l$. And $h$ is strictly convex if and  only if on  $l^+\setminus l$,
$h_+=h_-+\alpha h_l > h_-$, that is true if and only if $\alpha >0$.
\end{proof}

\begin{remark}\label{rem angle wedge}{\rm
The hyperplane $l$ in $B^d$ defines a  timelike vector hyperplane in Minkowski space, namely, if $B^d$ is identified with the Klein ball model of the hyperbolic space in $\R^{d,1}$,  the vector hyperplane passing through $l\times \{1\}$. Let $v_l$ be its unit spacelike normal vector pointing to the side containing $l^+$. Then it is easy to see that
\begin{equation}\label{vl}v_l=\frac{1}{p_l}\binom{n_l}{\|p_l\|}~,\end{equation} and so  the canonical map $h_l$ is the restriction to $B^d\times \{1\}$ of the  linear map $(x,x_{d+1})\mapsto\langle \binom{x}{x_{d+1}},v_l\rangle_{d,1} $. If $l$ is a vector hyperplane, then $v_l=\binom{n_l}{0}$.
Moreover, if $P_+$ and $P_l$ are the duals of the graphs of $h_+$ and $h_-$, then $P_+-P_-$ is colinear to $v$, that expresses the definition
\eqref{def angle} (compare also with Fact~\ref{fact vector dual}).
The absolute value of $\alpha$ is the Minkowski length of the spacelike segment $P_+-P_-$.  See Figure~\ref{fig dual wedge}.
}\end{remark}

\begin{fact}\label{fact h can isom}
Let $A\in O_+(d,1)$. Then
$h_{A\cdot l}=\frac{L}{L\circ A^{-1}}h_l\circ A^{-1}$.
\end{fact}
\begin{proof}
With the notations of Remark~\ref{rem angle wedge}, we clearly have
$v_{A\cdot l}=A(v_l)$, hence
$h_{A\cdot l}(x)=\langle \binom{x}{1},A(v_l)\rangle_{d,1}=\langle A^{-1}\binom{x}{1},v_l\rangle_{d,1}$, and by \eqref{A klein},
$A^{-1}(\binom{x}{1})=\frac{L(x)}{L(A^{-1}\cdot x)}\binom{A^{-1}\cdot x}{1}$.
\end{proof}

\begin{fact}\label{fact angle isom}
The image of the graph of a wedge by an orientation preserving co-Minkowski isometry is the graph of a wedge of same angle.
\end{fact}
\begin{proof}
The result is obvious from Remark~\ref{rem angle wedge}, as a co-Minkowski isometry acts as a Minkowski isometry on the dual objects, and hence sends a spacelike segment to a spacelike segment of same length.
\end{proof}

The choice of the normal $n_l$ gives an orientation on the vector hyperplane $l$, which is also  isometric to
$\H^{d-1}$. We denote by $\omega_l^{\mathbb{H}}$ its volume form for the hyperbolic metric.

\begin{lemma}\label{lem:mesure = angle}
Let $h$ be a convex wedge of angle $\alpha$ on a hyperplane $l$. Then the following identity holds:
$$\meanm(h)=\alpha\omega_l^{\mathbb{H}}~.$$
\end{lemma}

The simplest illustration of the lemma is
for $d=1$, $l=\{0\}$ and $h(x)=|x|=-x+2x\mathbf{1}_{\R_+}$. Then the angle is equal to $2$, and
$h''$ in the sense of distributions is equal to $2\delta(0)$.

\begin{proof}

From \eqref{def angle}, $h=h_-+\alpha h_l\mathbf{1}_{l^+}$, so as $h_-$ is affine, in the sense of distributions, $\partial_{ij}h=\alpha \partial_{ij}( h_l\mathbf{1}_{l^+})$. By successive integrations by part, for $\phi\in C^\infty_0$, using that $h_l=0$ on $l$ and that $h_l$ is affine, we obtain, in the sense of distributions,
$\partial_{ij}h=\alpha \partial_i h_l (n_l)_j\mbox{d}S$,  where $\mbox{d}S$ is  the (Euclidean) area form on $l$ ($n_l$ is an inward normal vector for $l^+$).

Hence by \eqref{eq measure mean} and \eqref{def angle}, the measure  $\meanm(h)$ is given by, for $x\in l$,
\begin{equation}\label{eq mesure angle}\alpha(\langle n_l,\operatorname{grad}h_l\rangle + \langle n_l,x\rangle\langle \operatorname{grad}h_l,x\rangle )L^{-d}(x)\operatorname{d}S(x)~.\end{equation}

Let us first consider that $l$ is the intersection of $B^d$ with a vector hyperplane.
Then,  $\langle n_l,x\rangle=0$, $\langle n_l,\operatorname{grad}h_l\rangle=1 $, and
from \eqref{eq:volume},  $L^{-d}(x)\operatorname{d}S=\operatorname{d}\omega_l^{\mathbb{H}}$. At the end of the day, \eqref{eq mesure angle} becomes
$\alpha  \operatorname{d}\omega_l^{\mathbb{H}},$
that is the wanted result when $l$ is defined by a vector hyperplane.
The general case follows by performing an orientation-preserving isometry sending $l$ to a vector hyperplane, and using Lemma~\ref{lem: action sur mesure} and Fact~\ref{fact angle isom}.
\end{proof}

\begin{figure}
\begin{center}
\psfrag{mink}{$\R^{d,1}$}
\psfrag{comink}{$\co$}
\psfrag{v}{$v_l$}
\psfrag{av}{$\alpha v$}
\psfrag{p}{$P_-$}
\psfrag{b}{$B^d$}
\psfrag{q}{$P_+$}
\psfrag{l}{$l$}
\psfrag{lt}{$\tilde{l}$}
\psfrag{o1}{$l_-$}
\psfrag{o2}{$l_+$}
\psfrag{ps}{$P_-^*$}
\psfrag{Qs}{$P_+^*$}
\psfrag{hp}{$h_-$}
\psfrag{hq}{$h_+$}

\includegraphics[width=0.8\hsize]{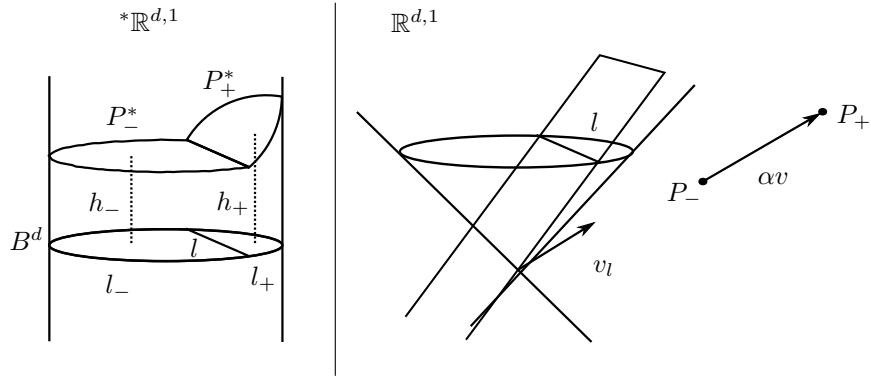}
\end{center}\caption{The spacelike segment in Minkowski space dual to a convex wedge in co-Minkowski space.}\label{fig dual wedge}
\end{figure}

\begin{remark}\label{remark christ}{\rm
Given a hyperplane $l$ of $B^d$ weighted by a positive number $\alpha$, it is almost clear how to construct a convex wedge
 in co-Minkowski space with angle $\alpha$. This construction can be easily extended to non intersecting weighted hyperplanes (see Section~\ref{simlicial}), or to a ``polyhedral case'', i.e. weighted hyperplanes are allowed to meet to form a convex cellulation of $\H^d$, together with a natural compatibility conditions at the weights, see  \cite[4.4]{fv} and \cite{FS} for the $d=2$ case. This is a polyhedral version of the \emph{Christoffel problem}, whose aim is to find a convex hypersurface in Minkowski space prescribing the dual Mean curvature measure --- called the \emph{area measure of order one} in this setting. The Christoffel problem in Minkowski space is the subject of \cite{fv}.

The polyhedral construction is also a version of the classical \emph{Maxwell-Cremona correspondence} or \emph{Maxwell lift}, see \cite{I-mc}.

}\end{remark}

\section{Action of cocompact hyperbolic isometry groups}

\subsection{Translation parts as cocycles}\label{space of cocycle}

Let $\Gamma$ be a subgroup of $O_+(d,1)$ such that
$\mathcal{H}^d/\Gamma$ is a compact oriented hyperbolic manifold.
A \emph{cocycle} $\tau\in Z^1(\Gamma, \R^{d,1})$  is a map $\tau:\Gamma\to \R^{d,1}$ satisfying, for $A,B\in \Gamma$,
 $$\tau(AB)=\tau(A) + A(\tau(B))~. $$
 Let us denote

$$\Gamma_\tau=\{(A,\tau(A)) | A\in \Gamma \}~. $$

From \eqref{group structure}, $\Gamma_\tau$ is a subgroup of the isometry group of  Minkowski space. In turn, it defines
a group of isometries of co-Minkowski space, that we will also denote by $\Gamma_\tau$.

In the cylindrical coordinates $B^d\times \R$ of co-Minkowski space,
$\Gamma$ acts freely and properly discontinuously on $B^d\times \{0\}$. As co-Minkowski space is the product manifold $B^d\times \R$, due to \eqref{action isom}, the following result is trivial, but worth to notice.

\begin{lemma}\label{lem:action propre}
The action of $\Gamma_\tau$ on $\co$ is free and properly discontinuous.
\end{lemma}

A \emph{coboundary} is a particular cocycle
of the form
$$\tau(A)=Av-v$$
for a given $v\in \R^{d,1}$. The group
$H^1(\Gamma,\R^{d,1})$ is the quotient of the space of cocycles by the space of coboundaries: two cocycles are in relation if and only if they differ by a coboundary.

In the following, we make the implicit assumption that
we are looking at $\Gamma$ such that  $H^1(\Gamma,\R^{d,1})$ is not reduced to zero.

Let us give a criterion of non-triviality.
Let us suppose that the compact hyperbolic manifold $\mathcal{H}^d/\Gamma$ contains $n$ disjoints embedded totally geodesic hypersurfaces $H_1,\ldots,H_n$.  Also, let us set some positive weights $\omega_i$ to each $H_i$.
This is actually a \emph{simplicial measured geodesic lamination} $\lambda$ on $\mathcal{H}^d/\Gamma$.

A lift to $B^d$ of a $H_i$ is a hyperplane
$l$. Recall from \eqref{vl} that a vector $v_l$ of $\R^{d,1}$ is assigned to any such $l$. Let us denote by $\tilde{L}$ the set of the lifts the $H_i$.
Let us fix an arbitrary base point $\tilde{x}\in B^d\setminus \tilde{L}$.
Then define, for $A\in \Gamma$, and for any path $c:[0,1]\to B^d$, transverse to $\tilde{L}$ and joining
$\tilde{x}$ to $A\cdot \tilde{x}$:
\begin{equation}\label{taulambda}\tau_\lambda(A)=\sum_{j \in c([0,1])\cap \tilde{L}}  \omega_j v_j~. \end{equation}

Clearly, the definition of $\tau_\lambda$ is independent from the choice of the path $c$ among paths transverse to $\tilde{L}$ joining the same endpoints.

\begin{fact}\label{fact: tau gamma 1}
With the notations above,  $\tau_\lambda\in Z^1(\Gamma,\R^{d,1})$.
\end{fact}

\begin{proof}
Let $A,B\in \Gamma$. Let $c_A,c_B:[0,1]\to B^d$ be
paths transverse to $\tilde{L}$, and joining $\tilde{x}$ to $A\cdot\tilde{x}$ and
$B\cdot \tilde{x}$ respectively. Let $c_{AB}$ be the concatenation of
$c_A$ with $A\cdot c_B$. This is a path joining $\tilde{x}$ to $(AB)\cdot \tilde{x}$ and transverse to $L$, so
$$\tau_\lambda(AB)=\sum_{j\in (c_{AB}([0,1])\cap\tilde{L})}\omega_jv_j
= \sum_{j\in (c_{A}([0,1])\cap\tilde{L})}\omega_jv_j  +\sum_{j\in (A\cdot c_{B}([0,1])\cap\tilde{L})}\omega_jv_j ~.$$

By definition of $v_l$, we clearly have $v_{A\cdot l}=A(v_l)$,  and  $A$ acts linearly on $\R^{d,1}$, so
$$\tau_\lambda(AB)=\sum_{j\in (c_{A}([0,1])\cap\tilde{L})}\omega_jv_j  +A\left(\sum_{j\in (c_{B}([0,1])\cap\tilde{L})}\omega_jv_j\right)=\tau(A)+A(\tau(B))~. $$

\end{proof}

\begin{fact}\label{fact: tau gamma 2}
Let $\tau'_\lambda$ be the cocycle defined by \eqref{taulambda}, but
choosing another basepoint $\tilde{x}'$. Then $\tau'_\lambda-\tau_\lambda$ is a coboundary.
\end{fact}
\begin{proof}
For any $A\in \Gamma$,
let $c:[0,1]\to B^d$ be a path transverse to $\tilde{L}$ joining $\tilde{x}$ to $A\cdot \tilde{x}$, and  let $c':[0,1]\to B^d$ be a path transverse to $\tilde{L}$ joining $\tilde{x}'$ to $A\cdot \tilde{x}'$. Let $\bar c:[0,1]\to B^d$ be any path transverse to $\tilde{L}$ joining $\tilde{x}$ to $\tilde{x}'$.
Then the concatenation $c^*$ of $\bar c$ with $c'$ and $-A\cdot \bar{c}$ is a transverse path joining $\tilde{x}$ to $A\cdot \tilde{x}$, so
$$\tau_\lambda(A)=\sum_{j\in (c^*([0,1])\cap\tilde{L})}\omega_jv_j =
\sum_{j\in (\bar{c}([0,1])\cap\tilde{L})}\omega_jv_j +\sum_{j\in (c'([0,1])\cap\tilde{L})}\omega_jv_j -\sum_{j\in (A\cdot\bar{c}([0,1])\cap\tilde{L})}\omega_jv_j $$
$$=
\sum_{j\in (\bar{c}([0,1])\cap\tilde{L})}\omega_jv_j +\tau'_\lambda(A) -A\left(\sum_{j\in (\bar{c}([0,1])\cap\tilde{L})}\omega_jv_j\right) $$
so if $v$ is the vector $-\sum_{j\in (\bar{c}([0,1])\cap\tilde{L})}\omega_jv_j$ we have $\tau_\gamma(A)-\tau'_\gamma(A)=Av-v$.
\end{proof}

So for each choice of positive weights, we have constructed an element of
$H^1(\Gamma,\R^{d,1})$.
Clearly, a linearly independent change in the weights will produce
a different element in $H^1(\Gamma,\R^{d,1})$, hence we have a simple geometric proof of the following classical result (see the Introduction).

\begin{theorem}
If $\H^d/\Gamma$ contains  $n$ disjoints embedded totally geodesic hypersurfaces, then the dimension of $H^1(\Gamma,\R^{d,1})$ is $\geq n$.
\end{theorem}

\subsection{Equivariant maps}

Let $\tau \in Z^1(\Gamma,\R^{d,1})$.
We will give more details on the action of $\Gamma_\tau$ by looking at particular functions. The analysis is simplified using the cylindrical coordinates of co-Minkowski space.
We say that a continuous map $h:B^d\to \R$ is \emph{$\Gamma$-invariant} if its graph is invariant for the action of $\Gamma$, i.e. for all $A\in \Gamma$, $(A,0)h=h$ (recall \eqref{def action h}):
$$\forall x\in B^d,\, (L^{-1} h)(A\cdot x)= (L^{-1}h)(x)~, $$
in other terms, $h$ is $\Gamma$-invariant if and only if
$L^{-1}h$ is invariant for the action of $\Gamma$.
In particular, if $h$ is $\Gamma$-invariant, as the action of $\Gamma$ is cocompact on $B^d$, $L^{-1}h$ is bounded.
Note that the function $L$  is obviously
 $\Gamma$-invariant (see Remark~\ref{rem inv tL} for a geometric viewpoint).

\begin{fact}\label{ext fuch}
Let $h$ be a $\Gamma$-invariant function. Then $h$ extends continuously as the constant zero function on $\partial B^d$.
\end{fact}
\begin{proof}
There exists two constants $c_1,c_2$ such that $c_1\leq L^{-1}h \leq c_2$, so $c_1L \leq h \leq c_2 L$, and the result follows.
\end{proof}

\begin{definition}\label{def equ} A continuous map $h:B^d\to \R$ is \emph{$\tau$-equivariant}
if its graph is invariant for the action of $\Gamma_\tau$, i.e. for all $A\in \Gamma$,
$(A,\tau(A))h=h,$
using the notation introduced in \eqref{def action h}.
\end{definition}

The vector space structure of $ Z^1(\Gamma, \R^{d,1})$ fits well with the vector space structure of maps, as the following lemma shows. Its proof is trivial from Definition~\ref{def equ}.

\begin{fact}\label{lem:diff}
Let $\tau_1,\tau_2\in Z^1(\Gamma, \R^{d,1})$ and let
$h_1$ and $h_2$ be $\tau_1$ and $\tau_2$-equivariant maps respectively, and $\alpha\in \R$.
Then $h_1+\alpha h_2$ is $(\tau_1+\alpha\tau_2)$-equivariant.
In particular, the difference between two $\tau$-equivariant map is a $\Gamma$-invariant map.
\end{fact}

\begin{fact}\label{fact unique cocycle}
If there are $\tau,\tau'\in  Z^1(\Gamma, \R^{d,1})$ such that
there is a map $\tau$ and $\tau'$-equivariant, then $\tau=\tau'$.
\end{fact}
\begin{proof}
For any $A\in \Gamma$ and any $x$, using the definition of equivariance, we obtain
$\langle A\cdot x,\overline{\tau(A)}\rangle_d-\tau(A)_{d+1}
= \langle A\cdot x,\overline{\tau'(A)}\rangle_d-\tau'(A)_{d+1}$.
\end{proof}

The following fact is clear from the definition of $\tau$-equivariant map and Lemma~\ref{lem:action hessien}.

\begin{fact}\label{lem:hessien inv}
Let $h:B^d\to \R$ be a $C^2$ $\tau$-equivariant function. Then
$L^{-1}\operatorname{Hess}h$ is $\Gamma$-invariant:
$$(L^{-1}\operatorname{Hess}h)(x)(X,Y)=(L^{-1}\operatorname{Hess}h)(A\cdot x)\left(DA(x)(X),DA(x)(Y)\right)~.$$
\end{fact}

\begin{remark}\label{remark cod quotient}{\rm
Fact~\ref{lem:hessien inv} says that the second fundamental form of the hypersurface which is the graph of $h$ (see Section~\ref{sec: extr}) defines a symmetric $(0,2)$-tensor on $\H^d/\Gamma$. Moreover this tensor is a symmetric Codazzi tensor, see Remark~\ref{rem II cod}.

}\end{remark}

It can be useful to note the following converse to Fact~\ref{lem:hessien inv}.

\begin{lemma}\label{difference hessien}
Let $h:B^d\to \R$ be a $C^2$ map such that $L^{-1}\operatorname{Hess} h$ is $\Gamma$-invariant. Then there exists a unique $\tau\in Z^1(\Gamma,\R^{d,1})$ such that $h$ is $\tau$-equivariant.
\end{lemma}
\begin{proof}
Let $A\in \Gamma$. As $A$ acts as an affine map on $B^d\times \R$, by the rule of the Hessian of a composition \eqref{hessien compose} and the invariance of the Hessian, we obtain

 $$\Hess (h \circ A) (x)(X,Y) = \hess h (A \cdot x)(DA(x)(X),DA(x)(Y))=\Hess h (x)(X,Y)~, $$
 hence $h$ and $h\circ A$ differ by an affine map, that in turn gives a vector $\tau(A^{-1})\in \R^{d,1}$:
 $$h(x)-h(A\cdot x)=\langle \tau(A^{-1}), \binom{x}{1}\rangle_-~.$$
 Writing $h(x)-h(A\cdot (B\cdot x))$ as
 $h(x)-h(B\cdot x) + h(B\cdot x) -h(A\cdot (B\cdot x))$, it follows that $\tau$ satisfies the cocycle relation. Uniqueness is given by Fact~\ref{fact unique cocycle}.
\end{proof}

Now let us  check that the discussion is not void. First there are easy examples in the coboundary case.

\begin{fact}\label{fact couboundary}
Let $\tau_v$ be a coboundary, i.e. there is $v\in \R^{d,1}$ such that $\tau_v(A)=Av-v$. Then
$h_v(x)=-\langle x,\bar{v}\rangle_{d}+v_{d+1}$ is a $\tau_v$-invariant map.
\end{fact}

In full generality, if the cocycle is equal to zero, we know the function $-L$  which is a  $C^\infty$  $\Gamma$-invariant function with positive definite Hessian. By the very general
 ``Ehresmann--Weil--Thurston holonomy principle'' \cite{gold}, for cocycles close to $0$ enough, there exist $\tau$-equivariant maps, which depends continuously on the cocycle. For convenience we recall the argument in our very simplified case, which follows the lines from
\cite[Lemma~I.1.7.2]{CEG}. We need to take care about convexity, that is also classical \cite{ghom}.

\begin{proposition}\label{prop exist smooth cvx}
For any cocycle $\tau$ there exists a $C^\infty$ convex (resp. concave) $\tau$-equivariant  function $h(\tau)$.

Moreover, if $\tau_n\to \tau$, then there exist $C^\infty$ convex (resp. concave)
$\tau_n$-equivariant functions
$h(\tau_n)$ such that $(h(\tau_n))_{n \in \N}$   converges to $h(\tau)$, and the second partial derivatives of $h(\tau_n)$ converge to the second partial derivatives of $h(\tau)$.
\end{proposition}

\begin{proof}
Clearly it suffices to prove the statement for the convex case.
Also by Fact~\ref{lem:diff}, it suffices to prove it for any cocycle close to $0$.

Let $\{B_i(r_i)\}_{i=1,\ldots,k}$ be
disjoint open balls of $\H^d$, such that $\Gamma \cdot \cup_i B_i(r_i)$ is a covering of $\H^d$.
 On $B_1(r_1)$, let us set $h_1=-L$. For $A\in \Gamma$ and $y\in A \cdot B_1(r_1)$,
let us set $h_1(y)=-L(y) + \langle \binom{y}{1}, \tau (A)\rangle_{d,1}$.
Such a function $h_1$ is $C^\infty$ and $\tau$-equivariant on $\Gamma \cdot B_1(r_1)$. The function $h_1$ converge to $-L$ uniformly on each orbit of $B_1(r_1)$ if $\tau$ goes to $0$.
Also the first partial derivatives of $h_1$ converge to
the ones of $-L$ uniformly on each orbit of $B_1(r_1)$ if $\tau$ goes to $0$. Moreover,
the Hessian of $h_1$ is equal to the one of $-L$ on $\Gamma \cdot B_1(r_1)$, in particular it is positive definite.

Let $r_i'<r_i$ for all $i$, such that
$\Gamma\cdot \cup_i B_i(r_i')$ is still a covering of $\H^d$.
Up to change the indices, suppose that $B_2(r_2)$ has non empty intersection with the orbit of $B_1(r_1)$.
Let $W$ be an open neighborhood of
$B_2(r_2') \cap \Gamma\cdot B_1(r_1')$
such that its closure is contained in $B_2(r_2)\cap \Gamma \cdot B_1(r_1)$. Let $\phi$ be a bump function
which is equal to $1$ on $B_2(r_2') \cap \Gamma\cdot B_1(r_1')$ and whose support is contained in $W$. Note that
the function $\phi h_1$ is well-defined and $C^\infty$ on $\H^d$, by setting the zero value out of $W$.

Let us define
 $f=\phi h_1 +  (1-\phi) (-L)$ on $B_2(r_2)$. The function $f$ is $C^\infty$, and equal to $h_1$ on $\Gamma\cdot B_1(r_1') \cap B_2(r_2')$.
  When the cocycle goes to $0$, $f$ and its first and second derivatives go to $-L$ and to its respective derivatives, uniformly on $B_2(r_2')$. In particular, we suppose that the cocycle is sufficiently small, so that the Hessian of $f$ is positive definite.

Then we define $h_2=f$ on $B_2(r_2')$, and by equivariance we define $h_2$ on $\Gamma \cdot B_2(r_2')$.
Also we set $h_2=h_1$ on
 $\Gamma \cdot B_1(r_1')$. By construction,
$h_2$ is well defined on
 the non-empty intersections between orbits of $B_1(r_1')$ and orbits of $B_2(r_2')$.
 Clearly, $h_2$ converges to $-L$ when the cocycle goes to
 $0$.
As the Hessian of $h_2$ converges to the one of $-L$ uniformly on $B_2(r_2')$, by Fact~\ref{lem:hessien inv}, this is true on each element in the orbit of $B_2(r'_2)$, in particular the Hessian of $h_2$ is positive definite.

 In the same way, if $r_i''<r_i'$ is such that
$\Gamma\cdot \cup_i B_i(r_i'')$ is still a covering of $\H^d$, then we can construct a function $h_3$, equivariant on the orbit of $B_1(r_1'')\cup B_2(r_2'')\cup B_3(r_3'')$ and satisfying the statement of the proposition. After a finite number of steps, we have constructed the wanted functions.

\end{proof}

\begin{corollary}\label{prop btau}
For any $\tau\in Z^1(\Gamma,\R^{d,1})$, there exists a continuous map $b_\tau:\partial B \to \R$ such that any $\tau$-equivariant map
extends continuously as $b_\tau$ on $\partial B$.

Moreover if  $\tau_1, \tau_2 \in Z^1(\Gamma,\R^{d,1})$ and $\alpha \in \R$, then $b_{\alpha \tau_1+\tau_2}=\alpha b_{\tau_1}+b_{\tau_2}$. And  $\tau$ is a coboundary if and only if $b_\tau$ is the restriction to $\partial B^d$ of an affine map of $\R^d$.
\end{corollary}
\begin{proof}
From Proposition~\ref{prop exist smooth cvx}, there exists
a $C^\infty$ convex $\tau$-equivariant map.
From Fact~\ref{lem:hessien inv}, $\mean(h)$ is a $\Gamma$-invariant function, hence bounded, so by
Proposition~\ref{lem Li}, there exists a continuous function
$b_\tau:\partial B \to \R$ that extends continuously $h$.
As the difference of two $\tau$-equivariant map is a $\Gamma$-invariant function, and as a $\Gamma$-invariant function extends continuously as the zero function on $\partial B$ (Fact~\ref{ext fuch}), it follows that $b_\tau$ is the continuous extension of any $\tau$-equivariant map.

The second property is obvious from the definition of $b_\tau$ and Fact~\ref{lem:diff}.
The last property follows from Fact~\ref{fact couboundary}
\end{proof}

From the existence of $b_\tau$ we deduce easily the existence of a unique $\tau$-equivariant \minmax map in the following lemma. The maps whose graphs are the boundary of the convex hull
of the graph of $b_\tau$ will be introduced in Section~\ref{sec vol}.

\begin{corollary}\label{lem equ minimal}
Let $\tau\in Z^1(\Gamma,\R^{d,1})$.
 There exists a unique
 $C^\infty$ $\tau$-equivariant map, denoted by $h^{\crit}_\tau$, satisfying $\mean(h^{\crit}_\tau)=0$.
 Moreover, for $\alpha \in \R$ and $\tau'\in Z^1(\Gamma,\R^{d,1})$,  $h^{\crit}_{\tau+\alpha\tau'}=
h^{\crit}_\tau+\alpha h^{\crit}_{\tau'}$, and $h^{\crit}_\tau$ is the restriction to $B^d$ of an affine map if and only if $\tau$ is a coboundary.
\end{corollary}
\begin{proof}

By Corollary~\ref{prop btau} and Proposition~\ref{prop critical gal}, we know that there exists a unique
 $C^\infty$ map, denoted by $h^{\crit}_\tau$, having $b_\tau$ as values on
 $\partial B$, and such that $\mean(h^{\crit}_\tau)=0$. This map is  $\tau$-equivariant. Indeed, apply an element of
 $\Gamma_\tau$ to the graph of $h^{\crit}_\tau$.
 Then we obtain the graph of a map with vanishing $\mean$ and boundary value $b_\tau$, so it has to be
  $h^{\crit}_\tau$ by uniqueness.

The second point is clear from Fact~\ref{lem:diff} and the fact that $h^{\crit}_\tau$ satisfies a linear equation, namely
$(\frac{1}{d}\Delta^{\H^d} - 1) h^{\crit}_\tau = 0$.
The last point is immediate from Corollary~\ref{prop btau}.
\end{proof}

\begin{remark}\label{rem: foliation CMC}{\rm
For any $t\in \R$, the map $h^{\crit}_\tau-tL$ is $\tau$-equivariant, with mean curvature equal to $t$. Hence the graphs of these maps gives a smooth foliation of $\co/\Gamma_\tau$ by hypersurfaces of constant mean curvature.
}\end{remark}

Corollary~\ref{lem equ minimal} allows to recover a classical relation between cocycles and traceless Codazzi tensors \cite{lafontaine,OS83,bsep}.
Let $\cod_0$ be the vector space
 of traceless symmetric Codazzi tensors on $\H^d/\Gamma$.
Let $\tau\in Z^1(\Gamma,\R^{d,1})$. By  Corollary~\ref{lem equ minimal}, there is a map $h^{\crit}_\tau$ whose second fundamental form is a $\Gamma$-invariant traceless Codazzi tensor (see Remark~\ref{remark cod quotient}), hence it defines an element of  $\cod_0$, denoted by $\Cod(\tau)$. By Corollary \ref{lem equ minimal}, the map $\Cod:Z^1(\Gamma,\R^{d,1}) \to \cod_0$ is linear. The kernel of this map corresponds to the $\tau$ such that $h^{\crit}_\tau$ is affine, hence to the coboundaries by Corollary \ref{lem equ minimal}. We thus obtain an injective morphism from
$H^1(\Gamma,\R^{d,1})$ to $\cod_0$, still denoted by $\Cod$.

\begin{proposition}\label{prop:ident H1 cod}
The map $\Cod:H^1(\Gamma,\R^{d,1})\to \cod_0$ is an isomorphism.
\end{proposition}
 \begin{proof}
 Let $C\in \cod_0$, that defines a $\Gamma$-invariant symmetric traceless Codazzi tensor $\tilde C$ on $\H^d$. By Lemma~\ref{lem tens cod}, there exists $h:B^d\to \R$ such that $\tilde C=\II_h$.
 From Lemma~\ref{difference hessien}, there exists a cocycle $\tau$ such that $h$ is $\tau$-equivariant, hence as $\II_h$ is traceless, by the uniqueness part of Corollary~\ref{lem equ minimal}, we will have $h=h^{\crit}_\tau$.

 \end{proof}

\subsection{Volume of the convex core and asymmetric norm}\label{sec vol}

\subsubsection{Convex core}

Let $\tau\in Z^1(\Gamma,\R^{d,1})$. There is an associated map $b_\tau:\partial B^d\to \R$ given by  Corollary~\ref{prop btau}. This map has a graph $\Lambda(b_\tau)$, and we will look at its
convex hull $\CH(\tau)$ in the affine space $\R^{d+2}$, as well as the functions $h^-_{b_\tau}$ and $h^+_{b_\tau}$ (see Section~\ref{sec convex hull}) whose graphs are the boundary  of $\CH(\tau)$. We will denote those two last maps  by $h^-_\tau$ and $h^+_\tau$ respectively.

The argument to check the following fact is analogous to the one used in the proof of Corollary~\ref{lem equ minimal}.

\begin{fact}
The map $h^-_\tau$ and $h^+_\tau$ are $\tau$-equivariant, in particular $\CH(\tau)$ is globally invariant for the action of $\Gamma_\tau$.

\end{fact}

\begin{lemma}\label{lem h+-}
Let $\tau\in Z^1(\Gamma,\R^{d,1})$. Then:
\begin{enumerate}[noitemsep]
\item for any convex (resp. concave) $\tau$-equivariant map $h$, then
$h\leq h^-_\tau$ (resp. $h\geq h^+_\tau$),
\item  $h^+_{\tau}=-h^-_{-\tau}$,
\item For $\alpha>0$, $h_{\alpha \tau}^-=\alpha h^-_\tau$,
\item $h^-_\tau + h^-_{\tau'} \leq h^-_{\tau+\tau'}$
and $h^+_\tau + h^+_{\tau'} \geq h^+_{\tau+\tau'}$.
\end{enumerate}
\end{lemma}
\begin{proof}
The two first points are from the definitions of $h^+_\tau$ and
$h^-_\tau$, Proposition~\ref{prop h+} and Corollary~\ref{prop btau}.
The third point follows from \eqref{eq:h-} and the fact that $b_{\alpha \tau}=\alpha b_\tau$.
The forth point follows from the first point,
as $h^-_\tau + h^-_{\tau'} $ is a convex $(\tau+\tau')$-equivariant function.
\end{proof}

\begin{lemma}\label{lem couboundary }
Let $\tau_v$ be a coboundary, then $h^-_{\tau_v}=h^{\crit}_{\tau_v}$ is an affine map and  $h^-_{\tau+\tau_v}=h^-_\tau+h^-_{\tau_v}$.
Conversely, if $h^-_\tau=h^{\crit}_\tau$, then $h^{\crit}_\tau$ is affine and $\tau$ is a coboundary.

\end{lemma}
\begin{proof}
If $\tau$ is a coboundary, we know that there exists a $\tau$-equivariant affine map (Fact~\ref{fact couboundary}).
Hence the convex hull of $\Lambda(b_\tau)$ is a piece of an hyperplane, and this hyperplane is also the $\tau_v$-\minmax hypersurface. Then $h^-_{\tau+\tau_v}=h^-_\tau+h^-_{\tau_v}$ follows from  \eqref{eq:h-} because $h_{\tau_v}$ is an affine map.
For the second part, on the one hand, $
\mean(h^{\crit}_\tau)=0$. On the other hand,   if $h^{\crit}_\tau=h^-_\tau$, then $h^{\crit}_\tau$ is convex, hence  affine (Remark~\ref{rem H positif}), so
 $b_\tau$ is the restriction to $\partial B^d$ of  an affine map. By Corollary~\ref{prop btau}, $\tau$ is a coboundary.

\end{proof}

\begin{definition} The \emph{convex core}
of $\co / \Gamma_\tau$, denoted by $\mathrm{CC}(\tau)$, is the smallest non-empty convex set of  $\co / \Gamma_\tau$.
\end{definition}

In the above definition, ``convex'' has to be understood in the strong sense of geodesically convex: $C$ is convex if for $x,y\in C$, any geodesic between $x$ and $y$ belongs to $C$. So for example, a single point or a small open ball may not be convex. In the cylindrical model of the universal cover, this notion of convexity coincides with the affine one.

Clearly, $\mathrm{CC}(\tau)=\CH(\tau)/\Gamma_\tau$. Hence $\co / \Gamma_\tau$ has a compact convex core, so the action of $\Gamma_\tau$ on $\co$ is \emph{convex cocompact}, in the sense of \cite{DGK1,DGK2}.

Recall the volume form on co-Minkowski space, Section~\ref{sec volume form}.
Let us denote by $\operatorname{Vol}$ the induced volume on
$^*\R^{d,1}$. It is then immediate than for any $\tau \in Z^1(\Gamma,\R^ {d,1})$,
\begin{equation}\label{eq vol cc}\operatorname{Vol}(\mathrm{CC}(\tau))=\int_{\mathbb{H}^d/\Gamma}  h^+_\tau - h^-_{\tau}~.
 \end{equation}

Here by abuse of notation, we denote in the same way the $\Gamma$-invariant function $h^+_\tau - h^-_{\tau}$ and the corresponding function on the compact hyperbolic manifold $\mathbb{H}^d/\Gamma$.
The integration is implicitly with respect to the volume form given by the hyperbolic metric.

\begin{definition}
The function $\vol: H^1(\Gamma,\R^{d,1})\to \R$
associates $\operatorname{Vol}(\mathrm{CC}(\tau))$ to any representative $\tau$ of an element of
  $H^1(\Gamma,\R^{d,1})$.
\end{definition}

By Lemma~\ref{lem couboundary }, $\vol$ is well-defined. Actually, the following result is straightforward to check
from Lemma~\ref{lem couboundary } and Lemma~\ref{lem h+-}.

\begin{proposition}
$\vol$ is a norm on $H^1(\Gamma,\R^{d,1})$.
\end{proposition}

\begin{remark}\label{rem flat ghmc}{\rm
The volume of the convex core
has the following geometric meaning in Minkowski space.
From a cocycle $\tau$, we have the boundary map $b_\tau$, that defines two convex sets $\Omega^+_{b_\tau}$ and $\Omega^-_{b_\tau}$ in Minkowski space, see Remark~\ref{rem domain of dependance}.
It follows from the previous section that those two sets (here denoted by
$\Omega^+_\tau$ and $\Omega^-_\tau$), are invariant under the action of $\Gamma_\tau$ on Minkowski space. Actually the action is free and properly discontinuous on $\Omega^+_\tau\cup \Omega^-_\tau$, and the quotient $\Omega_\tau^-$ (resp. $\Omega_\tau^+$) is a \emph{future complete flat} (resp. \emph{past complete}) Globally Hyperbolic Maximal Cauchy Compact (in short, GHMC)  spacetime. As the addition of a coboundary to the cocycle $\tau$ will only change the origin in Minkowski space, then $H^1(\Gamma,\R^{d,1})$ parametrizes the space of future complete (or past complete) flat GHMC spacetimes with a given linear holonomy, up to conjugacy. See \cite{Bar05,Bon05} for more details.

Moreover, for any cocycle $\tau$, we have that
$-\Omega^-_{\tau}=\Omega_{-\tau}^+$. Now, for any $x\in B^d$, let us denote by $\textbf{width}(x)$ the Lorentzian distance between the support plane
of $\Omega_{\tau}^+$ with outward unit normal $\binom{x}{1}$, and the support plane of $\Omega^-_{\tau}$ with inward unit normal $\binom{x}{1}$. Note that the map
$x\mapsto \mathbf{width}(x)$ is $\Gamma$-invariant. Then the mean width, defined as $\int_{\H^d/\Gamma}\mathbf{width}(\cdot)$, is given by \eqref{eq vol cc}, see Figure~\ref{fig:mean distance}.
}\end{remark}

\begin{figure}
\begin{center}
\psfrag{O+}{$\Omega_\tau^-$}
\psfrag{O}{$\Omega_\tau^+$}
\psfrag{x}{$x$}
\psfrag{h+-}{$|h_--h_+|(x)$}
\psfrag{h-}{$(x,h_+(x))$}
\psfrag{h+}{$(x,h_-(x))$}
\psfrag{h-e}{$(x,h_-(x))^*$}
\psfrag{h+e}{$(x,h_+(x))^*$}
\includegraphics[width=0.9\hsize]{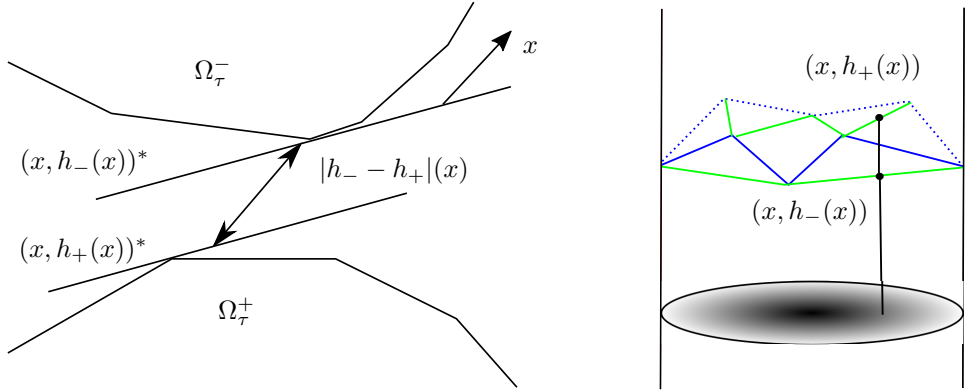}
\end{center}\caption{The volume of the convex core $\mathrm{CC}(\tau)$ is a ``mean distance'' between  a past complete flat GHMC spacetime and a future complete flat GHMC spacetime with the same holonomy.}\label{fig:mean distance}
\end{figure}

\subsubsection{Asymmetric norm}\label{sec asy}

In the previous section we showed that the volume of the convex core is a norm on $H^1(\Gamma,\R)$. We now see that it is actually the symmetrization of an asymmetric norm on
$H^1(\Gamma,\R)$.

For a cocycle $\tau$,
{the $S_1$ norm is defined as follows
\begin{equation}\label{eq asym}\|\tau \|_{S^1} = \int_{\H^d/\Gamma} h^{\crit}_\tau-h_\tau^-~. \end{equation}
(The denomination will be motivated in Remark~\ref{remS1}.)

By Lemma~\ref{lem couboundary }, if $\tau_v$ is a coboundary,
then $\|\tau+\tau_v\|=\|\tau\|+\|\tau_v\|=\|\tau \|$. Hence
$\|\cdot \|$ is well defined on   $H^1(\Gamma,\R^{d+1})$.

\begin{proposition}\label{prop:s1 norm}
 The $S_1$ norm $\|\cdot\|_{S_1}$ defines an asymmetric norm on $H^1(\Gamma,\R^{d+1})$, i.e. $\forall [\tau],[\tau']\in H^1(\Gamma,\R^{d+1})$
\begin{enumerate}[noitemsep]
 \item $\|[\tau]\|_{S_1}\geq 0$;
 \item $\|[\tau]\|_{S_1}=0$ if and only if $[\tau]=0$.
 \item $\|[\tau]+[\tau']\|_{S_1}\leq \|[\tau]\|_{S_1}+\|[\tau']\|_{S_1}$.
  \item\label{prop S1 3} $\forall \alpha \geq 0$, $\|\alpha [\tau]\|_{S_1}=\alpha \|[\tau]\|_{S_1}$;
\end{enumerate}
\end{proposition}
\begin{proof}
 The first property comes from Lemma~\ref{lem:crit conv}.
The second point is Lemma~\ref{lem couboundary }.
 The third and forth points are immediate consequence of Lemma~\ref{lem h+-}.
\end{proof}

It is obvious from  \eqref{eq asym} and \eqref{eq vol cc} that
 $\vol$ the symmetrization of $\|\cdot\|_{S_1}$:

$$\vol([\tau]) = \frac{1}{2}\left(\|[\tau]\|_{S_1}+\|-[\tau]\|_{S_1} \right)~. $$

\subsubsection{Mean curvature measure}

We now explain how $\|\cdot\|_{S^1}$  is related to the mean curvature measure introduced in Section~\ref{sec mean curv measure}.
From Lemma~\ref{lem: action sur mesure}, we have that for any convex $\tau$-equivariant map $h$, the measure $\meanm(h)$ is $\Gamma$-invariant, and then defines a Radon measure  $\meanm^{\Gamma}(h)$ on  $\H^d/\Gamma$.
Actually, there is a nice expression for this measure.
Let $h$ be a convex $\tau$-equivariant function.
By definition of $h^\crit_\tau$, $\meanm(h)=\meanm(h)-\meanm(h^\crit_\tau)=\meanm(h-h^\crit_\tau)$. On the other hand, $h-h^\crit_\tau$ is $\Gamma$-invariant, so we deduce easily that for any $C^\infty$ function $\varphi$ on $\H^d/\Gamma$,
$$\meanm^{\Gamma}(h)(\varphi)=\int_{\H^d/\Gamma}(h-h^\crit_\tau)(\frac{1}{d}\Delta^{\H^d}-1)\varphi~. $$

Taking $\varphi=1$,
$$\meanm^\Gamma(h)(\H^d/\Gamma)= \int_{\H^d/\Gamma} h_\tau^\crit - h~, $$
in particular, if $h=h^-_\tau$, by definition of the $S_1$ norm,
\begin{equation}\label{eq norme=measure}\|\tau \|_{S^1}=\meanm^\Gamma(h^-_\tau)(\H^d/\Gamma)~. \end{equation}

\begin{remark}\label{remS1}{\rm
Consider the convex set $\Omega_\tau^-$ in Minkowski space, as well as the
$\epsilon$-equidistant convex set $\Omega_\tau^-(\epsilon)$ (it is the dual convex set in Minkowski space of the epigraph of $h^-_\tau-\epsilon$ in $B^d\times \R$). By a Lorentzian version of the Steiner formula proved in \cite{fv}, the volume of
$(\Omega_\tau^- \setminus \Omega_\tau^-(\epsilon))/\Gamma_\tau$ is a polynomial in $\epsilon$ of degree $d+1$. Up to a dimensional constant, the coefficient in front of $\epsilon^d$ is nothing but $\meanm^\Gamma(h^-_\tau)(\H^d/\Gamma)$. The analoguous quantity in the classical theory of convex bodies is called the (total) area measure of order one \cite{schneider}, and usually denoted by $S_1$, that explains our terminology (see also Remark~\ref{remark christ}).

}\end{remark}

\begin{lemma}\label{prop conv coc}
Let  $\tau_n \to \tau$. Then $b_{\tau_n}$ (resp. $h^-_{\tau_n}$) pointwise converge to $b_\tau$ (resp. $h^-_{\tau}$).
\end{lemma}

\begin{proof}
By Proposition~\ref{prop exist smooth cvx}, we have convex (resp.  concave) $\tau_n$-equivariant functions converging to a $\tau$-equivariant convex (resp. concave) function. For any $n$, as the concave and the convex $\tau_n$-equivariant functions coincide
on $\partial B^d$ with $b_{\tau_n}$ given by Corollary~\ref{prop btau}, they bound a convex body $K_n$ of $\R^{d+1}$. Let us denote by $K$ the convex body bounded by the $\tau$-equivariant convex and concave functions.

Let us denote by $C^{d+1}$ the space of non-empty compact sets of $\R^{d+1}$, endowed with the Hausdorff topology.
Suppose that there is a subsequence $(K_{n_i})$ of $(K_n)$ that converges  to $K'$ in $C^{d+1}$. Then $K'$ is a convex body \cite[Theorem~1.8.6]{schneider}. Moreover, each point of $K'$ is the limit of a sequence  of points $(x_{n_i})$ with $x_{n_i}\in K_{n_i}$ \cite[Theorem~1.8.8]{schneider}. From this it is easy to deduce that   $K'=K$.

Now as the $\tau_n$-equivariant functions are converging, they are bounded, and in turn the sequence of convex bodies $(K_n)_n$ is bounded in $B^d\times \R\subset \R^{d+1}$.
By the Blaschke selection theorem \cite[Theorem~1.8.7]{schneider}, there is a subsequence $K_{n_i}$ converging to a convex body $K'$.
Moreover, the sequence $K_n$ is contained in a compact subspace of
$C^{d+1}$ \cite[Theorem~1.8.4]{schneider}. As we saw that any convergent subsequence of $(K_n)$ converges to $K$, it follows that $(K_n)$ converges to $K$.

As the limit of any  convergent sequence $(x_{n_i})$ with $x_{n_i}\in K_{n_i}$ must belong to $K$ \cite[Theorem~1.8.8]{schneider}, it is easy to deduce that
$b_{\tau_n}\to b_\tau$. This easily implies the Hausdorff convergence of
 $\CH(\tau_n)$ to  $\CH(\tau)$, see e.g. \cite[Lemma~2.1]{graham2}, which in turn gives the convergence of $h^-_{\tau_n}$ to $h^-_\tau$, as the Hausdorff
 convergence of convex bodies implies the Hausdorff convergence of the boundaries \cite[Lemma 1.8.1]{schneider}.
\end{proof}

\begin{remark}{\rm
By standard properties of convex functions, it follows from Lemma~\ref{prop conv coc} that $h^-_{\tau_n}$ converges to $h^-_\tau$ uniformly on any compact sets of $B^d$. But one cannot deduce a uniform convergence of $b_{\tau_n}$
from the pointwise convergence \cite{overflow}.

Actually we will obtain the uniform convergence of the $h^-_{\tau_n}$ as a byproduct of the considerations of Section~\ref{sec:anosov}.
In particular, Lemma~\ref{le:tauuniform} will imply the following proposition, without mention to the mean curvature measure.

}\end{remark}

\begin{proposition}\label{prop continuite seminorme}
The $S_1$ norm $\| \cdot \|_{S_1}:H^1(\Gamma,\R^{d,1})\to \R$ is continuous.
\end{proposition}
\begin{proof}
Let $\tau_n\to \tau$. From Lemma~\ref{prop conv coc},  $h^-_{\tau_n}$ converges to  $h^-_{\tau}$, and, using a partition of the unity, it is not hard to deduce from
Lemma~\ref{lem conv measures}
that $\meanm^\Gamma(h^-_{\tau_n})$ weakly converge to  $\meanm^\Gamma(h^-_{\tau})$, so that the result follows from \eqref{eq norme=measure}.

\end{proof}

\subsubsection{Simplicial measured geodesic laminations}\label{simlicial}

We use the notations and definitions of Section~\ref{space of cocycle}, where we have considered a simplicial measured geodesic lamination $\lambda$ on the compact hyperbolic manifold $\H^d/\Gamma$. Namely we have supposed that $\H^d/\Gamma$ contains $n$ disjoints embedded totally geodesic hypersurfaces $H_1,\ldots,H_n$ with
positive weights $\omega_i$.

Let us push the construction a step forward.
For any $y\in B^d$, let $c:[0,1]\to B^d$ be any curve transverse to
$\tilde{L}$ joining the base point $\tilde{x}$ to $y$, and define
$$h_\lambda(y)=\sum_{j\in c[0,1]\cap \tilde{L}} \mu_j h_{l_j}(y) $$
where $h_{l_j}$ is the canonical map associated with $l_j$ (Definition~\ref{def can map}).

\begin{fact} If $\tau_\lambda$ is the cocycle given by \eqref{taulambda}, then
$h_\lambda=h^-_{\tau_\lambda}$.
\end{fact}
\begin{proof}
As the weights are positive, by Fact~\ref{fact convex wedge}, $h_\lambda$ is a convex map.

Let us check that $h_\lambda$ is $\tau_\lambda$-equivariant. Let $\tilde{c}:[0,1]\to B^d$ be a path joining $\tilde{x}$ to $A\cdot \tilde{x}$, and let $c':[0,1]\to B^d$ be a path joining $\tilde{x}$ to $y$, both assumed to be transverse to $\tilde{L}$. Then the concatenation of $c$ with $A\cdot c'$ is a path joining $\tilde{x}$ to $A\cdot y$, hence
$$h_\lambda(A\cdot y)= \sum_{j\in (c([0,1])\cap\tilde{L})}\omega_jh_{l_j}(y)+\sum_{j\in (A\cdot c'([0,1])\cap\tilde{L})}\omega_jh_{l_j}(y)~,$$
and as by definition $h_l(y)=\langle \binom{y}{1},v_l\rangle_{d,1}$,
then $\sum_{j\in (c([0,1])\cap\tilde{L})}\omega_jh_{l_j}(y)=\langle \binom{y}{1}, \tau_\lambda(A)\rangle_{d,1}$. Also,
$\sum_{j\in (A\cdot c'([0,1])\cap\tilde{L})}\omega_jh_{l_j}(y)=\sum_{j\in ( c'([0,1])\cap\tilde{L})}\omega_jh_{A\cdot l_j}(A\cdot y)$, and by
Fact~\ref{fact h can isom}, $h_{A\cdot l_j}(A\cdot y)=\frac{L(A\cdot y)}{L(y)}h_l(y)$. The equivariance is proved.

A $h_\lambda$ is a convex $\tau_\gamma$-equivariant map, then $h_\lambda\leq h^-_{\tau_\lambda}$. By construction, the graph of $h_\lambda$ is made of segments joining
points of graph of $b_{\tau_\lambda}$, hence it is contained in $\CH(b_{\tau_\lambda})$, so $h_\lambda\geq h^-_{\tau_\lambda}$.

\end{proof}

The \emph{length} $\length(\lambda)$ of a simplicial
measured geodesic lamination $\lambda$
on $\H^d/\Gamma$ is defined as sum of the weights times the total volume of the corresponding totally geodesic hypersurfaces. By Lemma~\ref{lem:mesure = angle} and \eqref{eq norme=measure}, we obtain the following.

\begin{proposition}\label{prop length simplicial}
Let $\lambda$ be a simplicial measured geodesic lamination on $\H^d/\Gamma$.
Then $$\length(\lambda)=\|\tau_\lambda\|_{S^1}~.$$
\end{proposition}

\begin{remark}{\rm
There is no reason why for $d\geq 3$  any cocycle should arise from a (simplicial) measured geodesic lamination on $\H^d/\Gamma$.  So for $d\geq 3$, the concept of measured geodesic lamination is not sufficient. A more suitable concept is the one of \emph{measured geodesic stratification},  introduced in \cite{Bon05}. In contrast, we will see in the next section that for $d=2$, any cocycle arises from a measured geodesic lamination.
}\end{remark}


\subsection{The case of dimension $2+1$}\label{cas 2+1}

In this part we study the particularities of the $d=2$ case.
We will denote by $\teich$ the Teichm\"uller space of a compact surface
homeomorphic to $\H^2/\Gamma$.  We will denote by
 $g$ the genus of $S$.

\subsubsection{Goldman isomorphism}

The Teichm\"uller space $\teich$ can be defined  as the space of faithful and discrete representations of $\pi_1S$ into  $\operatorname{Isom}_0(\mathbb{H}^2)$ up to conjugacy.  Let $\rho$
be such a representation  such that $\Gamma=\rho(\pi_1S)$. Then the tangent space of $\teich$ at $\rho$ naturally identifies with
$H^1(\pi_1(S),\mathfrak{isom}(\mathbb{H}^2))$, where
$\pi_1S$ acts on the Lie group $\mathfrak{isom}(\mathbb{H}^2)$ via $\operatorname{Ad}\rho$ \cite{gol84}. Using the hyperboloid model $\mathcal{H}^2$
for $\mathbb{H}^2$, $\mathfrak{isom}(\mathbb{H}^2)$ can be identified with
$\mathfrak{o}(2,1)$. Let us write it as follows.

\begin{theorem}[{\cite{gol84}}]
There is a vector space isomorphism
$$\gold: H^{1}(\Gamma,\mathfrak{o}(2,1))  \to T_{\H^2/\Gamma}\teich~.$$
\end{theorem}

There is also a one-to-one correspondence between vectors of $\R^{2,1}$ and infinitesimal Minkowski isometries of $\mathfrak{o}(2,1)$ ---this may be seen for example using the Minkowski cross product, see e.g. \cite{DG}. This identification gives a vector space isomorphism
$$\mathbf{C} : H^1( \Gamma,\R^{2,1})\to H^1(\Gamma,\mathfrak{o}(2,1))~, $$
and in turn we have the following vector space isomorphism
$$\xi=\gold \circ \mathbf{C} : H^1( \Gamma,\R^{2,1}) \to  T_{\H^2/\Gamma}\teich~.$$

In particular, we obtain the following.

\begin{corollary}\label{cor dim H}
The vector space $H^1(\Gamma,\R^{2,1})$ has dimension $6g-6$.
\end{corollary}

\subsubsection{Mess homeomorphism}

Let us call an \emph{entire segment} of $B^2$ a segment of $B^2$
whose endpoints are in $\partial B^2$.
A \emph{geodesic lamination} $\tilde{L}$ of $B^2$ is a non-empty closed union of
disjoint entire segments of $B^2$.
Let $\tilde{L}$ be a geodesic lamination on $B^2$ which is invariant under the action of $\Gamma$. Then the image $L$ of $\tilde{L}$ under the projection is a geodesic lamination on the compact hyperbolic surface
$B^2/\Gamma$.
A \emph{measured geodesic lamination} $\lambda=(L,\mu)$ on $B^2/\Gamma$ is the data of a geodesic lamination $L$ together with a \emph{transverse measure} $\mu$, that is the data of a Radon measure
on each compact rectifiable curve transverse to $L$, such that
\begin{itemize}[nolistsep]
\item the support of the measure is the intersection of the arc with $L$,
\item if two arcs are homotopic through arc transverse to $L$, then the
homotopy sends the measure on one segment to the measure on the other one.
\end{itemize}

A simplicial measured geodesic lamination on $B^2/\Gamma$ is a set of non-intersecting closed simple geodesics weighted by positive numbers. Note that the action of $\Gamma$ onto $B^2$ is via the identification of the disc with the Klein model of the hyperbolic plane, but the notation $B^2$ stands for reminding the affine nature of measured geodesic lamination on the universal cover.

Let $\mathcal{M}\mathcal{L}_\Gamma$ be the set of measured geodesic laminations on the compact hyperbolic surface $\H^2/\Gamma$.
 $\mathcal{M}\mathcal{L}_\Gamma$ is endowed with the following topology.
We say that $\lambda_n$ converge to $\lambda$ if, for any compact segment $c$ transverse to $L$ then
\begin{itemize}[nolistsep]
\item $c$ is transverse to $L_n$ for $n$ big,
\item $\mu_n$ weakly converge to $\mu$ on $c$.
\end{itemize}

We have the following classical result of Thurston, see e.g. \cite{Bon01} and the references therein.

\begin{theorem}[Thurston]\label{thm dim lam}
For the topology defined above, $\mathcal{M}\mathcal{L}_\Gamma$  is a manifold of dimension $6g-6$.
\end{theorem}

Recall from \eqref{vl} that a vector $v_l$ of $\R^{2,1}$ is assigned to any entire segment $l$ of $B^2$. Let $\mathbf{e}$ be a continuous function such that, for any path $c:[0,1]\to B^2$ transverse to $L$,
$\mathbf{e}_L(c(t))=v_l$ is $c(t)\in l$ and $l\in \tilde{L}$. Let us fix an arbitrary base point $\tilde{x}\in B^2$. Then define, for $A\in \Gamma$, and for any path $c:[0,1]\to B^2$ transverse to $L$ joining
$\tilde{x}$ and $A\cdot \tilde{x}$:
\begin{equation}\tau_\lambda(A)=\int_0^1 \mathbf{e}_L(c(t)) \mbox{d}\mu(t)~. \end{equation}

As the measure is transverse, the definition of $\tau_\lambda$ is independent from the choice of the path $c$ and the function $\mathbf{e}_L$. The following fact is proved formally in the same way as Facts \ref{fact: tau gamma 1} and \ref{fact: tau gamma 2}.

\begin{fact}
We have $\tau_\lambda\in Z^1(\Gamma,\R^{2,1})$. Moreover, if the basepoint is changed, the new cocycle differ from the preceding one by a coboundary.
\end{fact}

Hence
we have constructed a well defined map
$$\mess : \mathcal{M}\mathcal{L}_\Gamma\to H^1(\Gamma,\R^{d,1})~.$$

\begin{theorem}[{\cite{Mes07}}] \label{thm mess}
The map $\mess$ defined above is a homeomorphism.
\end{theorem}
\begin{proof}
The map is clearly  injective and continuous.
By Theorem~\ref{thm dim lam} and Corollary~\ref{cor dim H}, both
 $\mathcal{M}\mathcal{L}_\Gamma$ and  $H^1(\Gamma,\R^{d,1})$ are manifolds
  of same dimension. Hence $\mess$
 is a local homeomorphism by the invariance of domain theorem.
Now for $\lambda\in \mathcal{M}\mathcal{L}_\Gamma$ and $t\geq 0$, let us define $t\lambda$ as the measured geodesic lamination obtained from $\lambda$ by simply multiplying the transverse measure by $t$. By \eqref{taulambda}
we clearly have $\mess(t\lambda)=t\mess(\lambda)$. As $H^1(\Gamma,\R^{2,1})$ is a vector space and $\mess$ a local homeomorphism, it follows that $\mess$ is surjective.
\end{proof}

\begin{remark}{\rm
We could have given a direct proof of Theorem~\ref{thm mess}, by defining a ``bending measure'' belonging to  $\mathcal{M}\mathcal{L}_\Gamma$ from the graph of $h^-_\tau$, for any cocycle $\tau$.
There are at least three ways to define such a bending measure. The first one is to mimic the construction of the bending measure given by the upper boundary component of the convex core of a hyperbolic quasi-Fuchsian manifolds  \cite{EM}.
The second one is to define, as in \cite{Mes07}, the induced distance on the spacelike part of the boundary of $\Omega_\tau^-$, the dual of the epigraph of $h^-_\tau$ in Minkowski space. The last one is to consider the mean curvature measure given by $h^-_\tau$.
}\end{remark}

\subsubsection{Length of measured geodesic laminations}

We have encountered the length of simplicial measured geodesic laminations in Section~\ref{simlicial}.
For $d=2$, the length of a
measured geodesic lamination
is defined as the total mass on
the surface of the measure which is the product of the hyperbolic measure along the
leaves of the lamination and the measure transverse to the leaves. We refer to \cite{Bon01} for more details. Actually,
the simplicial case suffices, as the following results shows.
 One may see for example
 Lemma 2.4 in \cite{ker} for the first one, and
 Theorem~3.1.3 in \cite{HP} or Section~3.4.3 in \cite{BB09} for the second one.

\begin{proposition}
The map $\length:\mathcal{M}\mathcal{L}_\Gamma\to \R$ is continuous.
\end{proposition}

\begin{proposition}
Simplicial measured geodesic laminations are dense in $\mathcal{M}\mathcal{L}_\Gamma$.
\end{proposition}

So from the above results, Proposition~\ref{prop length simplicial} generalizes as follows.

\begin{proposition}\label{prop:length norm}
Let $\lambda\in \mathcal{M}\mathcal{L}_\Gamma$. Then
$$\length(\lambda)=\|\mess(\lambda)\|_{S^1}~.$$
\end{proposition}

\subsubsection{Thurston earthquake norm}

From a measured geodesic lamination $\lambda$ on
 $\H^2/\Gamma$, one obtains another hyperbolic metric on $S$ by performing a (left) earthquake along the lamination ---we refer to \cite{ker} and the reference therein for more details about earthquakes. Actually for $t$ near $0$, earthquakes along $t\lambda$ define a path in $\teich$ starting at $\H^2/\Gamma$. This path has a well defined derivative at $0$, which gives an element in $ T_{\H^2/\Gamma}\teich$, the tangent space of Teichm\"uller space at the point $\H^2/\Gamma$. In turn, we have an \emph{infinitesimal earthquake map}:

 $$\earth: \mathcal{M}\mathcal{L}_\Gamma \to  T_{\H^2/\Gamma}\teich~. $$

\begin{theorem}[{\cite[Proposition~2.6]{ker}}]\label{thm ker1}
The map $\earth$ is a homeomorphism.
\end{theorem}

So the map $ \earth\circ\mess^{-1}$ provides a homeomorphism between $H^1(\Gamma,\R^{2,1})$ and  $T_{\H^2/\Gamma}\teich$. Although there is no natural vector space structure on $\mathcal{M}\mathcal{L}_\Gamma$, we have the following.

\begin{proposition}[{\cite[Proposition B.3]{BS12}}]\label{prop BS}
We have that  $\earth\circ\mess^{-1}=\xi$. In particular, $\earth\circ\mess^{-1}$ is a vector space isomorphism.
\end{proposition}

In other terms, as $\xi=\gold \circ \mathbf{C}$, the following diagram commutes:
$$\begin{tikzcd}
\mathcal{M}\mathcal{L}_\Gamma \arrow{r}{\earth} \arrow[swap]{d}{\mess} & T_{\H^2/\Gamma}\teich \\
H^1(\Gamma,\R^{2,1})  \arrow{r}{\mathbf{C}} & H^1(\Gamma,\mathfrak{o}(2,1)) \arrow[swap]{u}{\gold}
\end{tikzcd}~.$$

\begin{definition}
Let $X\in T_{\H^2/\Gamma}\teich$. The \emph{earthquake norm} of $X$ is
$$\|X\|_{\operatorname{earth}}=\length(\earth^{-1}(X))~.$$
\end{definition}

From Proposition~\ref{prop:length norm} and Proposition~\ref{prop BS}, one has in fact
$$\|X\|_{\operatorname{earth}}=\|\xi^{-1} (X) \|_{S_1} $$
and as $\xi$ is a vector space isomorphism, from Proposition~\ref{prop:s1 norm}, one finally obtains the following result.

\begin{theorem}[{\cite[Theorem~5.2]{thu98}}]
The earthquake norm is an asymmetric norm on $T_{\H^2/\Gamma}\teich$.
\end{theorem}

%
%
%
%
%
%

\begin{remark}{\rm
There is a smooth analogue of Proposition~\ref{prop BS} proved in \cite{bsep}. Namely, Proposition~\ref{prop:ident H1 cod} gives a map
$\Cod$ from $H^1(\Gamma,\R^{d,1})$ to $\cod_0$, the space of traceless symmetric Codazzi tensors on $\H^2/\Gamma$. In dimension $2$, there is also an isomorphism $\infdef$ from $\cod_0$ to $T_{\H^2/\Gamma}\teich$, where a $(0,2)$-tensor is seen as an
infinitesimal deformations of the hyperbolic metric (see \cite{tromba}, where such tensors are called TT, for \emph{transverse traceless}).
Then, if $J$ is the almost complex structure of $T_{\H^2/\Gamma}\teich$, the following diagram commutes:

$$
\begin{tikzcd}
\cod_0 \arrow{r}{\infdef} \arrow[swap]{d}{\Cod^{-1}} & T_{\H^2/\Gamma}\teich\arrow{d}{J} \\
 H^1(\Gamma,\R^{2,1}) \arrow{r}{\xi} &   T_{\H^2/\Gamma}\teich
\end{tikzcd}~.
$$

}\end{remark}

\subsubsection{Thurston length norm}

Following \cite{thu98}, we note that two successive identifications of the tangent space $T_{\H^2/\Gamma}\teich$ of Teich\"muller space with the cotangent space $T_{\H^2/\Gamma}^*\teich$ will permit to define another asymmetric norm on  $T_{\H^2/\Gamma}\teich$, that is actually the Thurston length norm (see the Introduction).

A first identification between
$T_{\H^2/\Gamma}^*\teich$  and $T_{\H^2/\Gamma}\teich$ is given by the Weil--Petersson form of Teichm\"uller space, that is a symplectic form on $T_{\H^2/\Gamma}\teich$. For
$\alpha \in  T_{\H^2/\Gamma}^*\teich$, let $\alpha^\sharp$ be the dual element in
$T_{\H^2/\Gamma}\teich$ of $\alpha$ for the symplectic form. Then define naturally

$$\| \alpha \|^*_{\operatorname{length}}:=\| \alpha^\sharp\|_{\operatorname{earth}}~. $$

On the other hand, for any vector space $E$ endowed with an asymmetric norm $N$, then its dual $E^*$ is endowed with the asymmetric norm $N^*$ defined by, for $v\in E^*$,

$$N^*(v):=\sup\left\{\frac{v(x)}{N(x)} |  x\in E\setminus \{0\}  \right\}~. $$

Applying this to the cotangent space of Teichm\"uller space endowed with $\|\cdot \|^*_{\operatorname{length}}$, we obtain a new asymmetric norm on the tangent space of Teichm\"uller space.

\begin{definition}
Let $X\in T_{\H^2/\Gamma}\teich$. The \emph{length norm} of $X$ is
$$\|X\|_{\operatorname{length}}=\sup\left\{\frac{\alpha(X)}{\|\alpha^\sharp\|_{\operatorname{earth}}} |  \alpha\in T_{\H^2/\Gamma}^*\teich\setminus \{0\}  \right\}~.$$
\end{definition}

If it possible to describe more precisely the length norm, using a famous result of Wolpert.
Let $\lambda$ be a measured lamination on the surface $S$. The function
$\length(\lambda)$ on the Teichm\"uller space of $S$ is defined as follows: for each choice of a hyperbolic metric on $S$,  $\length(\lambda)$ is the length of the corresponding measured geodesic lamination.
Due to a formula of Wolpert \cite[Lemma~4.1]{wol}, the tangent vector $\earth(\lambda)$ of the Teichm\"uller space of $S$ at a point $\H^2/\Gamma$ is the
symplectic gradient of the function $\length(\lambda)$ at the same point, with respect to
the Weil--Petersson form of Teichm\"uller space:
\begin{equation*}
\mbox{d}\length(\lambda)^\sharp=\earth(\lambda)~,
\end{equation*}
in particular,
$$\|\mbox{d}\length(\lambda)\|^*_{\operatorname{length}}=\length(\lambda)~. $$

Theorem~\ref{thm ker1} together with Wolpert result gives an identification between the cotangent space of Teichm\"uller space and $\mathcal{M}\mathcal{L}_\Gamma$ \cite[Lemma~2.3]{KS}.
In consequence we obtain the following.

\begin{theorem}[{\cite[Theorem~5.1]{thu98}}]
Let $\alpha\in T_{\H^2/\Gamma}^*\teich$, and $\lambda$ such that $\alpha=\operatorname{d}\length(\lambda)$. Then
$$\|\alpha\|^*_{\operatorname{length}} =\length(\lambda)$$
defines an asymmetric norm on $T_{\H^2/\Gamma}^*\teich$.
\end{theorem}

And finally, the length norm on  $T_{\H^2/\Gamma}\teich$ writes as follows: for
$X\in T_{\H^2/\Gamma}\teich$,
\begin{equation}\label{eq length norm}\|X\|_{\operatorname{length}}=\sup\left\{\frac{\operatorname{d}\length (\lambda)(X)}{\length(\lambda)}| \lambda\in \mathcal{M}\mathcal{L}_\Gamma\setminus \{0\}  \right\}~.\end{equation}

\section{Anosov representations}\label{sec:anosov}

In all this section, $\Gamma$ is a cocompact lattice of $O_+(d,1)$, and
$\tau$ an element of $ Z^1(\Gamma, \R^{d,1})$. We will consider the associated group $\Gamma_\tau$ of isometries of
co-Minkowski space. The aim of this section is to provide an alternative proof (Proposition \ref{pro:btau}) of the existence and uniqueness of
the $\tau$-invariant map $b_\tau$ already exhibited in Lemma \ref{prop btau}. This proof involves the
\textit{Anosov character} of $\Gamma_\tau$ as a representation of the hyperbolic group $\Gamma$ into the group
of isometries of the Minkowski space.
As a by-product, we will see
that the convergence in Lemma \ref{prop conv coc} is not only pointwise, but uniform (Lemma \ref{le:tauuniform}).

We start by the following fundamental observation: since stabilizer of points are non-compact, there is no $O_+(d,1) \ltimes \R^{d,1}$ invariant metric
on the boundary of co-Minkowski space.  However, if one fixes an element $x_0$ in $\mathcal{H}^d$,
then $x \mapsto \langle x,x \rangle_{d,1} + 2\langle x,x_0\rangle_{d,1}^2$ is a positive definite form,
hence an Euclidean metric on $\R^{d,1}$, that we denote by $\langle \cdot,\cdot \rangle_{x_0}$.

The choice of $x_0$ also induces a splitting $\R^{d,1} \approx \R^d \times \R$: here, $\R$ is
the linear subspace spanned by $x_0$, and $\R^d$ is the orthogonal of $x_0$ for the Minkowskian scalar product. Until now, when writing
$\partial\co \approx \partial B^d \times \R$, we  were always implicitly doing the choice $x_0 = (0, \ldots, 0, 1)$,
but in this section we will also consider other choices. What is relevant for us now, is that
the choice of $x_0$ induces a Riemannian metric $g_{x_0}$ on $\partial\co$:
the one making $\partial B^d$ and $\R$ orthogonal, and whose restrictions to $\partial B^d$ and $\R$ are the ones
induced by the Euclidean metric $\langle \cdot,\cdot \rangle_{x_0}$.

Let us be more precise: we can define $\partial\co$ as the space of lightlike affine hyperplanes of Minkowski space. Once fixed the unit timelike vector $x_0$, we can parametrize $\partial\co$ by pairs $(w, h)$ where:
\begin{itemize}
  \item $w$ is a future lightlike vector in $\R^{d,1}$ in the affine spacelike hyperplane $H_{x_0}$ of equation $\langle x_0, w\rangle_{d,1}=-1$ (therefore, $H_{x_0}$ is the hyperplane tangent to
      $\mathcal{H}^d$ at $x_0$),
  \item $h$ any real number.
\end{itemize}
The associated lightlike affine hyperplane is then the one given by the equation:
$$\langle w, \cdot \rangle_{d,1} = -h~.$$

The set of future lightlike vectors lying in $H_{x_0}$ is the unit sphere in this Euclidean space.
The metric $g_{x_0}$ is the product of the usual metric on this unit sphere by the usual metric
on the real line.
The distance function we will actually use is not the one induced by $g_{x_0}$
but the following one:
\begin{eqnarray*}
  d_{x_0}((w_1, h_1), (w_2, h_2)) &=& \sqrt{\langle w_1-w_2, w_1-w_2\rangle_{d,1} + (h_1-h_2)^2} \\
   &=& \sqrt{-2\langle w_1, w_2\rangle_{d,1} + (h_1-h_2)^2}~.
\end{eqnarray*}

We will also consider the closed hyperbolic manifold $N = \Gamma\backslash\mathcal{H}^d$, and the
 geodesic flow $\phi^t$ on the unitary tangent bundle $M = T^1N$. Recall that, for any element $v$ of
 $M$, the image $\phi^t(v)$ is the unique vector tangent to the geodesic starting from $v$ and at distance $t$
 along this geodesic.

 Actually, $M$ is the quotient of the unitary tangent bundle $T^1\mathcal{H}^d$ by the natural action of $\Gamma$.
 The unitary tangent bundle $T^1\mathcal{H}^d$ is also naturally identified with pairs $(x,v)$, where the base point
 $x$ is an element of $\mathcal{H}^d$, and $v$ a unit spacelike vector in Minkowski space orthogonal to $x$. The geodesic
 flow $\widetilde{\phi}^t$ on $T^1\mathcal{H}^d$ is then:
 $$\widetilde{\phi}^t(x,v) = (\cosh(t)x + \sinh(t)v, \sinh(t)x + \cosh(t)v)~.$$

\begin{definition}[Foliated bundle over $M$]\label{def:foliated}
  Let $E_\tau$ be the quotient of the product $T^1\mathcal{H}^d \times \partial\co$ by the diagonal action
  of $\Gamma_\tau$ ---where  $\Gamma_\tau$ acts on $\mathcal{H}^d$ through its linear part. Let $\pi_\tau: E_\tau \to M$ be the map induced by the projection
  on the first factor. This map is a fibration, of fiber $\partial\co/\Gamma_\tau$.
  It is called the foliated bundle of
  holonomy group $\Gamma_\tau$ over $M$.
\end{definition}

\begin{definition}[Lifted geodesic flow]\label{def:liftedflow}
  Let $\widetilde{\phi}_\tau^t$ be the flow on $T^1\mathcal{H}^d \times \partial\co$ defined by:
  $$\widetilde{\phi}^t((x,v), \xi) = (\widetilde{\phi}^t(x,v), \xi)~.$$

  This flow commutes with the $\Gamma_\tau$ action, and induces a flow on $E_\tau$, denoted by
  $\phi^t_\tau$.
\end{definition}

We clearly have:
$$\forall t \in \R \;\; \phi^t_\tau \circ \pi_\tau = \pi_\tau \circ \widetilde{\phi}^t~.$$

We also can distinguish two subbundles $\Delta_\tau^\pm$ of $\pi_\tau: E_\tau \to M$. More precisely:

\begin{lemma}\label{le:-1}
 Let $(x,v)$ in $T^1\mathcal{H}^d$. Let $(w,h)$ be an element of $\partial\co$ parametrized by the pair $(w,h)$
 under the identification defined above associated to $x$. Then:
 $$-1 \leq \langle w, v \rangle_{d,1} \leq 1~.$$
 Moreover, the equality $\langle w, v \rangle_{d,1} = 1$ holds if and only if $w=x+v$, and the equality
 $\langle w, v \rangle_{d,1} = -1$ holds if and only if $w=x-v$.
\end{lemma}

\begin{proof}
  For every  $(x,v)$ in $T^1\mathcal{H}^d$, and every   lightlike element $w$ of $H_x$, $v$, $-v$ and
  $w-x$ are unit elements in the Euclidean hyperplane $x^\perp$. The Lemma follows easily since
  $\langle w, v \rangle_{d,1} = \langle w-x, v \rangle_{d,1}$.
\end{proof}

\begin{definition}
We denote by $\widetilde{\Delta}^+$ (respectively $\widetilde{\Delta}^-$) the closed subset of
$T^1\mathcal{H}^d \times \partial\co$ comprising elements $(x,v, \xi)$ such that the orthogonal of the
lightlike hyperplane $\xi$ is $x+v$ (respectively $x-v$).

The complement $T^1\mathcal{H}^d \times \partial\co \setminus \widetilde{\Delta}^\pm$
is an open subset that we denote by $\widetilde{\aleph}^\pm$.
\end{definition}

It is straightforward to check that $\widetilde{\Delta}^\pm$ and $\widetilde{\aleph}^\pm$ are $\Gamma_\tau$-invariant and define closed subsets $\Delta_\tau^\pm$ and open subsets $\aleph_\tau^\pm$ of
$E_\tau$. Moreover:

\begin{lemma}\label{le:alephinvariant}
$\widetilde{\Delta}^\pm$ and $\widetilde{\aleph}^\pm$ are $\widetilde{\phi}^t$-invariant.
\end{lemma}

\begin{proof}
  We just have to prove that $\widetilde{\Delta}^\pm$ is $\widetilde{\phi}^t$-invariant. We just treat
  the case of $\widetilde{\Delta}^+$, the case of $\widetilde{\Delta}^-$ is similar.

  Let $(x,v,\xi)$ be an element of $\widetilde{\Delta}^+$: it means that, for the parametrization defined
   by $x$, the lightlike hyperplane $\xi$ is parametrized by $(w,h)$, where $w=x+v$ ---or, equivalently,
   $\langle w, v \rangle_{d,1}=1$ (see Lemma \ref{le:-1}). Denote by $(x_t,v_t)$ the iterate
   $\widetilde{\phi}^t(x,v)$. Let $(w_t,h_t)$ be the pair parameterizing $\xi$ for the identification defined by $x_t$. Then, $w_t = \lambda_t(x+v)$ for some positive real number $\lambda_t$. We must have:
   \begin{eqnarray*}
     -1 &=& \langle w_t, x_t \rangle_{d,1} \\
      &=& \langle \lambda_t(x+v), \cosh(t)x + \sinh(t)v \rangle_{d,1} \\
      &=& -\lambda_t\cosh(t)+\lambda_t\sinh(t)\\
      &=& -\lambda_t\exp(-t)~.
   \end{eqnarray*}
   Therefore $\lambda_t=\exp(t)$, and:

   \begin{eqnarray*}
     w_t &=& \exp(t)(x+v)\\
         &=& (\cosh(t)+\sinh(t))x + (\cosh(t)+\sinh(t))v\\
         &=& (\cosh(t)x + \sinh(t)v) + (\sinh(t)x + \cosh(t)v) \rangle_{d,1} \\
         &=& x_t + v_t~.
   \end{eqnarray*}

The Lemma follows.

\end{proof}

Therefore, $\Delta^\pm_\tau$ are $\phi^t_\tau$-invariant.
The restriction of $\pi_\tau$
to $\Delta_\tau^\pm$ is a fibration, with $1$-dimensional fibers. The restriction $\pi_\tau^\pm$ to $\aleph_\tau^\pm$
is a fibration with contractible fibers. Indeed, every fiber is the complement in $\partial\co$ of a degenerate vertical line removed, i.e. the product of a $1$-punctured sphere by the real line.

\begin{definition}
  Let $\digamma_\tau$ be the space of continuous sections of the fibration $\pi_\tau: E_\tau \to M$.
  We equip $\digamma_\tau$ with the following metric:
  $$D(\sigma_1, \sigma_2) = \sup_{(x,v) \in T^1\mathcal{H}^d} d_x(\sigma_1(x,v), \sigma_2(x,v))~.$$
  We denote by $\digamma^\pm_\tau$ the open subset comprising sections of $\pi_\tau: \aleph^\pm_\tau \to M$, and
  by $\digamma(\Delta^\pm)_\tau$ the space of sections of $\pi_\tau^\pm: \Delta_\tau^\pm \to M$.
\end{definition}

Since $M$ is compact, this upper bound is always attained.
Observe that the metric space $(\digamma_\tau, D)$
is complete.

Let $\sigma$ be an element of $\digamma_\tau$. It lifts to a $\Gamma_\tau$-equivariant section of
the fibration $T^1\mathcal{H}^d \times \partial\co \to T^1\mathcal{H}^d$ and therefore provides
a $\Gamma_\tau$-equivariant maps $F: T^1\mathcal{H}^d \to \partial\co$. Actually, $\digamma_\tau$ is in
$1$-$1$ correspondence with the space of  $\Gamma_\tau$-equivariant maps from $T^1\mathcal{H}^d$ into $\partial\co$.

The flow $\phi_\tau$ induces a $1$-parameter group of transformations on $(\digamma_\tau, D)$:
for every $t$ in $\R$, and any $\sigma$ in $\digamma_\tau$, define:
$$\Phi^t_\tau(\sigma) (x,v) = \phi^{t}_\tau (\sigma(\phi^{-t}(x,v)))~.$$
According to Lemma \ref{le:alephinvariant}, the subbundles $\digamma^\pm_\tau$ are $\Phi^t_\tau$-invariant.

We can now prove the fundamental fact:

\begin{lemma}\label{le:Dcontract}
  The flow $\Phi^t_\tau$ on $\digamma^+_\tau$ is exponentially contracting: there exist positive real numbers $T$, $a$ and $0<C<1$ such that, for every $t>T$ and for every $\sigma_1$, $\sigma_2$ in $\digamma^+_\tau$ we have:
  $$D(\Phi^t_\tau(\sigma_1), \Phi^t_\tau(\sigma_2)) < Ce^{-at} D(\sigma_1, \sigma_2)~.$$
\end{lemma}

\begin{proof}
Let $F: T^1\mathcal{H}^d \to \partial\co$ be a $\Gamma_\tau$-equivariant map
corresponding to elements of $\digamma_\tau^+$.
We denote by $F_t$ the iterate $\Phi^t_\tau(F)$.
Let $(x,v)$ be an element of $T^1\mathcal{H}^d$.  Let $\xi$ be the image $F(x,v)$. It
is an affine lightlike hyperplane. By definition of $\Phi^t_\tau$, $\xi$ is the image under
$F_t$ of $\widetilde{\phi}^t(x,v) = (x_t, v_t)=(\cosh(t)x + \sinh(t)v, \sinh(t)x + \cosh(t)v)$.
Let $(w_t, h_t)$ be the pair corresponding to $\xi$ satisfying $\langle x_t, w_t\rangle_{d,1}=-1$
and such that $\xi$ is the hyperplane of equation:
$$\langle w_t, . \rangle_{d,1} = -h_t~.$$

In particular, we see that $-hx$ belongs to $\xi$, and therefore, for every $t$ we have:
\begin{equation}\label{eq:h}
  h_t = -h\langle w_t, x \rangle_{d,1}~.
\end{equation}

Since the lightlike vectors $w_t$ are all orthogonal to $\xi$, they are proportional:
for every $t$, there is a real number $\lambda_t > 0$ such that $w_t = \lambda_tw_0$.
From equation \eqref{eq:h} we see:
$$h_t = h\lambda_t~.$$

A straightforward computation shows:
$$\lambda_t = \frac{1}{\cosh(t) - \sinh(t)\langle v, w_0 \rangle_{d,1}}~.$$


Let now $F_1$, $F_2$ be two $\Gamma_\tau$-equivariant maps  from $T^1\mathcal{H}^d$ into $\partial\co$ corresponding to sections of $\pi_\tau: \aleph_\tau^+ \to M$.
  The distance in $\digamma_\tau$ between the corresponding  sections is then the supremum of $d_x(F_1(x,v), F_2(x,v))$
  where $(x,v)$ describes $T^1\mathcal{H}^d$. Applying $\Phi^t_\tau$ simply means that we replace
  $F_1$ and $F_2$ by $F_1 \circ \widetilde{\phi}_\tau^{-t}$ and $F_2 \circ \widetilde{\phi}_\tau^{-t}$. It follows that
  the distance after applying $\Phi^\tau$ is the supremum of $d_{x_t}(F_1(x,v), F_2(x,v))$ where $(x,v)$ describes $T^1\mathcal{H}^d$ and \textit{where $x_t$ denotes as above the $x$ component of $\widetilde{\phi}^t(x,v)$,} i.e. $\cosh(t)x + \sinh(t)v$.

  The computation above shows that, for $i=1,2$, the pair $(w^i_t, h^i_t)$ representing  $F_i(x,v)$ satisfy:
  \begin{eqnarray*}
    w^i_t &=&  \frac{w^i_0}{\cosh(t) - \sinh(t)\langle v, w_0 \rangle_{d,1}}~,\\
    h^i_t &=& \frac{h^i_0}{\cosh(t) - \sinh(t)\langle v, w_0 \rangle_{d,1}}~.
  \end{eqnarray*}

  Therefore, $d_{x_t}(F_1(x,v), F_2(x,v)) = \frac{d_{x_0}(F_1(x,v), F_2(x,v))}{\cosh(t) - \sinh(t)\langle v, w_0 \rangle_{d,1}}$.
  Since the $F_1$ and $F_2$ correspond to sections in $\digamma^+_\tau$ , we have:
  $$-1 \leq \langle v, w_0 \rangle_{d,1} < 1~.$$

It follows that for big $t$, the quantity $\cosh(t) - \sinh(t)\langle v, w_0 \rangle_{d,1}$ is equivalent
to $e^t(1-\langle v, w_0 \rangle_{d,1})/2$.
The Lemma follows.
\end{proof}

\begin{corollary}\label{cor:section}
  There exists one and only one $\Phi^t_\tau$-invariant section $\sigma^+_\tau$ of $\pi_\tau: \aleph_\tau^+ \to M$.
  This invariant section actually takes value in $\Delta_\tau^-$.
\end{corollary}

\begin{proof}
  Let $T>0$ be a real number big enough so that $\Phi^T_\tau$ is contracting.
  Since
  $\Delta_\tau^-$ is a subbundle of $\aleph_\tau^+$, $\digamma(\Delta^-)_\tau$ is a closed subset of $\digamma^+_\tau$. The restriction $D$ to $\digamma(\Delta^-)_\tau$ is therefore complete. Hence, as any
  contracting map acting on a complete metric space, $\Phi^T_\tau$ admits a unique fixed point $\sigma^+_\tau$ in $\digamma(\Delta^-)_\tau$.
  Since its action on $\digamma^+_\tau$ is contracting too, $\sigma^+_\tau$ is the unique fixed point in $\digamma^+_\tau$. Since $\Phi^T_\tau$ commutes with $\Phi^t_\tau$ for every real number $t$, $\sigma^+_\tau$
  is fixed by every $\Phi^t_\tau$.
\end{proof}

Let $F_\tau: T^1\mathcal{H}^d \to \partial\co$ be the $\Gamma_\tau$-equivariant lifting of the $\Phi_\tau^t$-invariant section $\sigma^+_\tau$ exhibited in Corollary \ref{cor:section}. The $\Phi_\tau^t$-invariance
means that $F_\tau$ is constant along the orbits of the geodesic flow $\widetilde{\phi}^t$ of $T^1\mathcal{H}^d$.
The following Lemma shows that we have much more:

\begin{lemma}\label{le:constantu}
The map $F_\tau$ is constant along the leaves of the weak unstable foliation of the geodesic flow $\phi^t$.
\end{lemma}

\begin{proof}
  Let $\theta_1$, $\theta_2$ be two orbits of $\widetilde{\phi}^t$ in the same \textit{unstable leaf}, i.e. such that
for every $(x_1, v_1)$ in $\theta_1$ and every $(x_2, v_2)$ in $\theta_2$ the isotropic vectors
$x_1-v_1$ and $x_2-v_2$ are proportional, i.e. represent the same element of $\partial\mathcal{H}^d$.
On the other hand, since the invariant section takes value in $\Delta_\tau^-$, we have that $F_\tau(x_1,v_1)$ and $F_\tau(x_2,v_2)$ are lightlike hyperplanes orthogonal to respectively $x_1-v_1$ and $x_2-v_2$. Therefore, they are
parallel.

Let $p_1$, $p_2$ be the projections of  $(x_1, v_1)$ and $(x_2, v_2)$
in $M$. Then, by replacing $p_2$ by another element of its $\phi^t$-orbit, one can assume that $p_1$ and $p_2$
lies in the same \textit{strong unstable leaf}, i.e. that the hyperbolic distance between $\phi^t(p_1)$ and $\phi^t(p_2)$ converge exponentially to $0$ when $t$ goes to $-\infty$.

It follows that the hyperbolic distance between $\widetilde{\phi}^{t}(x_1, v_1)$ and
$\widetilde{\phi}^{t}(x_2, v_2)$ converges to $0$ when $t$ is going to $-\infty$. Let $\xi_1=F_\tau(x_1, v_1)$ and
$\xi_2 = F_\tau(x_2, v_2)$. Since $F_\tau$ is
(uniformly) continuous, it follows that $d_{t}(\xi_1, \xi_2)$ converges to $0$,
where $d_{t}$ is the distance on $\partial\co$ defined by $\widetilde{\phi}^{t}(x_1, v_1)$.
But this is almost a contradiction with Lemma \ref{le:Dcontract}, that shows that this distance should be exponentially
increasing when $t$ is going to $-\infty$. The only possibility is that this distance is actually vanishing,
i.e. $\xi_1=\xi_2$. The Lemma is proved.

\end{proof}

In the sequel, we use the cylindrical affine model
of the co-Minkowski space, i.e. write elements of $\partial\co$
as pairs $(w,h)$ where $w$ is a lightlike vector satisfying $\langle x_0,w \rangle_{d,1}=-1$, where
$x_0$ denotes the element $(0, \ldots, 0, 1)$ of $\R^{d,1}$.

\begin{proposition}\label{pro:btau}
  There is a continuous map $b_\tau: \partial B^d \to \R$ such that the $\Gamma_\tau$-equivariant map
   $F_\tau: T^1\mathcal{H}^d \to \partial\co$ is given by:
$$(x,v) \mapsto \left(-\frac{x-v}{\langle x_0,x-v \rangle_{d,1}}, b_\tau\left(-\frac{x-v}{\langle x_0,x-v \rangle_{d,1}}\right)\right)~.$$
\end{proposition}

\begin{proof}
We still parameterize
the unit tangent bundle of the hyperbolic space by pairs $(x,v)$ where $x$ is a unit timelike vector and $v$ a unit spacelike vector orthogonal to $x$.

Since the invariant section takes value in the subbundle $\Delta^-$, the map $F_\tau$
is such that $F_\tau(x,v) = (w(x,v), h(x,v))$ where $w(x,v)$ is proportional
to $x-v$, hence is equal to $-\frac{x-v}{\langle x_0,x-v \rangle_{d,1}}$. Moreover, according to Lemma \ref{le:constantu}, $h(x,v)$ depends only on $x-v$, hence, only on $-\frac{x-v}{\langle x_0,x-v \rangle_{d,1}}$. Therefore,
$F_\tau$ is given by:
$$(x,v) \mapsto \left(-\frac{x-v}{\langle x_0,x-v \rangle_{d,1}}, b_\tau\left(-\frac{x-v}{\langle x_0,x-v \rangle_{d,1}}\right)\right)$$
for some map $b_\tau: \partial B^d \to \R$.
\end{proof}

As a corollary, we get the following amelioration of Lemma \ref{prop conv coc}:

\begin{lemma}\label{le:tauuniform}
Let  $\tau_n \to \tau$. Then $b_{\tau_n}$ (resp. $h^\pm_{\tau_n}$, $h_{\tau_n}^\crit$) converge uniformly to $b_\tau$ (resp. $h^\pm_{\tau}$, $h^\crit_\tau$).
\end{lemma}

\begin{proof}
  We just give a sketch of proof. First, we observe that we just have to prove the statement for $b_{\tau_n}$,
  since the uniform convergence of $h^\pm_{\tau_n}$ (resp. $h^\crit_\tau$) follows then from Lemma \ref{le:btau} (resp. Lemma~\ref{conv min}). The key point
  is that when $n$ is big enough, the fibration $\pi_{\tau_n}: E_{\tau_n} \to M$ is isomorphic to
  the fibration $\pi_{\tau}: E_{\tau} \to M$. More precisely, (the inverse of) this isomorphism of fibrations
  send the graph of the section $\sigma_{\tau}$ to the graph of some section which is already a nearly fixed point
 for $\Phi^t_{\tau_n}$. The bigger $n$ is, the closer (for the metric $D$) is this nearly fixed point
 to the eventual fixed point $\sigma^+_{\tau_n}$. In other words, the bigger is $n$, the closer to
 $\sigma^+_\tau$ is $\sigma^+_{\tau_n}$ for the compact-open topology. The Lemma clearly follows, due to
 the form of the lifts $F_\tau$ and $F_{\tau_n}$ given by Proposition \ref{pro:btau}.
\end{proof}

\begin{remark}\label{rk:repeler}
{\rm \emph{Mutandi mutandis,} one can show that there is also a unique fixed point for $\Phi^t_\tau$ in $\digamma^-_\tau$, which is
this time an exponential repeller, and which is actually a section of the subbundle $\Delta^+$. It provides,
as in Proposition \ref{pro:btau} a map from $\partial B^d$ into $\R$, which is actually
the map $b_\tau$. Details are left to the reader. }
\end{remark}

\begin{remark}{\rm
Instead of considering the fiber bundle $\aleph^\pm_\tau$, one might have restricted the study to
the subbundles $\Delta^\pm$, which are simpler since with one-dimensional fibers. However, the most efficient way
to deal with these bundles is to consider them as subbundles of $E_\tau$.
}\end{remark}

We conclude this section by an interpretation of its content in term of Anosov representations.
Let $G$ be a general Lie group acting on some space $X$, and let $\rho: \Gamma \to G$ be a representation.
Consider as in Definition \ref{def:foliated} the foliated bundle $\pi_\rho: E_\rho(X) \to M$ where $E_\rho(X)$
is the quotient of the product $T^1\mathcal{H}^d \times X$ by the diagonal action of $\Gamma$ ---where the
action of $\Gamma$ on $X$ is given by $\rho$. As in Definition \ref{def:liftedflow}, the geodesic flow
$\phi^t$ lifts to some horizontal flow $\phi^t_\rho$ on $E_\rho(X)$ so that the bundle map $\pi_\rho$ is equivariant.

The representation $\rho$ is said \emph{$(G,X)$-Anosov} if the following holds: there
is a section $\sigma: M \to E_\rho(X)$ which is equivariant for the flows, and such that the graph $\Lambda$
of $\sigma$ is a \emph{closed hyperbolic subset} for the lifted flow $\phi^t_\rho$: it means
that the restriction $T_\Lambda E_\rho(X)$ of the tangent bundle of $E_\rho(X)$ to $\Lambda$ splits as a Whitney sum
of subbundles $E^+ \oplus E^- \oplus \Phi$, where:
\begin{itemize}
  \item $\Phi$ is the one dimensional bundle tangent to the flow $\phi^t_\rho$,
  \item $E^+$ is exponentially contracted by the flow,
  \item $E^-$ is exponentially expanded by the flow.
\end{itemize}
For more details, see \cite{labourie} or \cite{barflag, bardefor}.

In our case, the inclusion $\Gamma \approx \Gamma_\tau \subset SO_+(d,1) \ltimes \R^{d,1}$ is
$(G,X)$-Anosov where $X$ is the space of oriented $(d-1)$-dimensional spacelike affine subspaces of
$\R^{d,1}$. Indeed, $X$ identifies with the open domain in $\partial\co \times \partial\co$ made of
pairs $(\xi_1, \xi_2)$, where $\xi_1$ and $\xi_2$ are non parallel affine lightlike hyperplanes. Therefore, the two
equivariant sections $\sigma^\pm_\tau$ define altogether a section $\sigma$ of $\pi_\rho: E_\rho(X) \to M$. Moreover,
it follows from Lemma \ref{le:Dcontract} and Remark \ref{rk:repeler} that the graph of $\sigma$ is
a closed hyperbolic subset for $\phi^t_\rho$.

\begin{spacing}{0.9}
\begin{footnotesize}
\bibliography{mesure-S1}
\bibliographystyle{alpha}
\end{footnotesize}
\end{spacing}

\end{document}